\documentclass{article}
\usepackage[utf8]{inputenc}

\usepackage{amsthm}
\usepackage{amsmath}
\usepackage{amssymb}
\usepackage{multirow}
\usepackage{mathtools}
\usepackage{enumitem}
\usepackage{pgf,tikz,pgfplots}
\usepackage{mathrsfs}

\usepackage{xcolor}
\usepackage{graphicx}
\usepackage{caption}
\usepackage{wrapfig}
\usepackage{float}
\usepackage[justification=centering]{caption}
\usepackage{geometry}

\usepackage{subcaption}
\usepackage[hidelinks]{hyperref}
\usepackage{cleveref}

\pgfplotsset{compat=1.15}
\usetikzlibrary{arrows}
\usetikzlibrary{circuits.logic.US,circuits.logic.IEC,fit}

\usetikzlibrary{arrows.meta, 
calc, decorations.markings, positioning, shapes.geometric}
\tikzset{baseline={($ (current bounding box.west) - (0,1ex) $)}, auto}
\tikzset{vertex/.style={circle, inner sep=1.5pt, fill}, edge/.style={thick, line join=bevel}}

\definecolor{gray}{rgb}{0.4,0.4,0.4}
\definecolor{black}{rgb}{0,0,0}
\definecolor{red}{rgb}{0.8,0,0}
\definecolor{blue}{rgb}{0,0,1}
\definecolor{pur}{rgb}{.75, 0.2, .85}
\definecolor{gold}{rgb}{0.99,0.8,0.01}

\usetikzlibrary{calc,graphs,fit}

\newcommand{\mygraph}[1]{\tikz[baseline={($ (current bounding box.west) - (0,1ex) $)}]{\pgftransformscale{0.4} \graph [empty nodes, nodes={circle, inner sep=1.2pt, fill=black}, edges={thick, line to}, no placement, left anchor=, right anchor=] {#1};}}

\definecolor{purple}{RGB}{102,0,222} 

\definecolor{jade}{rgb}{0,0.66,0.42}

\DeclarePairedDelimiter\abs{\lvert}{\rvert}%

\newtheorem{thm}{Theorem}[section]
\newtheorem{lem}[thm]{Lemma}

\newtheorem{cor}[thm]{Corollary}

\newtheorem{prop}[thm]{Proposition}
\theoremstyle{definition}
\newtheorem{defn}[thm]{Definition}
\newtheorem{example}[thm]{Example}

\newtheorem{rem}[thm]{Remark}

\title{Iterated Jump Graphs}
\author{Fran Herr and Legrand Jones II}
\makeatletter
\def\@maketitle{%
  \newpage
  \null
  \vskip 2em%
  \begin{center}%
  \let \footnote \thanks
    {\LARGE \@title \par}%
    \vskip 1.5em%
    {\large
      \lineskip .5em%
      \begin{tabular}[t]{c}%
        \@author
      \end{tabular}\par}%
    \vskip 1em%
  \end{center}%
  \par
  \vskip 1.5em}
\makeatother

\begin{document}

\maketitle

\begin{abstract}
    The jump graph $J(G)$ of a simple graph $G$ has vertices which represent edges in $G$ where two vertices in $J(G)$ are adjacent if and only if the corresponding edges in $G$ do not share an endpoint. In this paper, we examine sequences of graphs generated by iterating the jump graph operation and characterize the behavior of this sequence for all initial graphs. We build on work by Chartrand et al. who showed that a handful of jump graph sequences terminate and two sequences converge. We extend these results by showing that there are no non-trivial repeating sequences of jump graphs. All diverging jump graph sequences grow without bound while accumulating certain subgraphs.
\end{abstract}
        
\section{Introduction}\label{Section 1}

Over the course of this paper, we will be studying sequences of graphs generated by the jump graph operation. Given a simple graph $G$, its jump graph $J(G)$ has vertices that represent edges of $G$; two vertices in $J(G)$ are connected by an edge if and only if the corresponding edges of $G$ are not incident. For readers familiar with graph theory, $J(G)$ is the complement of the line graph $L(G)$. In this paper, we completely characterize the end behavior of any graph under iteration of the jump graph operation.

\begin{figure}[H]
  \centering
  \begin{minipage}[b]{0.3\textwidth}\centering    \begin{tikzpicture}[line cap=round,line join=round,>=triangle 45,x=0.6cm,y=0.6cm]
    \draw [line width=1pt] (-1,2)-- (-1,0);
    \draw [line width=1pt] (-1,0)-- (1,0);
    \draw [line width=1pt] (2,2)-- (2,0);
    \draw [line width=1pt] (-1,2)-- (1,2);
    \draw [line width=1pt] (1,2)-- (1,0);
    \begin{scriptsize}
    \draw [fill=black] (-1,2) circle (1.8pt);
    \draw (-1.3,1) node {2};
    \draw (0,-0.3) node {3};
    \draw (1.3,1) node {4};
    \draw (0,2.3) node {1};
    \draw (2.3,1) node {5};
    \draw [fill=black] (-1,0) circle (1.8pt);
    \draw [fill=black] (1,2) circle (1.8pt);
    \draw [fill=black] (1,0) circle (1.8pt);
    \draw [fill=black] (2,2) circle (1.8pt);
    \draw [fill=black] (2,0) circle (1.8pt);
    \end{scriptsize}
    \end{tikzpicture}
    \caption*{\large $\Gamma$}
  \end{minipage}
    \begin{minipage}[b]{0.3\textwidth}\centering
    \begin{tikzpicture}[line cap=round,line join=round,>=triangle 45,x=0.6cm,y=0.6cm]
    \draw [line width=1pt] (-1,2)-- (-1,0);
    \draw [line width=1pt] (0.5,1)-- (2,0);
    \draw [line width=1pt] (2,2)-- (2,0);
    \draw [line width=1pt] (-1,0)-- (0.5,1);
    \draw [line width=1pt] (-1,2)-- (0.5,1);
    \draw [line width=1pt] (0.5,1)-- (2,2);
    \begin{scriptsize}
    \draw (-1,2.5) node {1};
    \draw (-1,-0.5) node {3};
    \draw (0.5, 0.5) node {5};
    \draw (2,-0.5) node {4};
    \draw (2,2.5) node {2};
    \draw (-1.4,1) node {a};
    \draw (-0.3,1.9) node {b};
    \draw (-0.3,0.2) node {c};
    \draw (2.4,1) node {f};
    \draw (1.2,1.9) node {d};
    \draw (1.2,0.2) node {e};
    \draw [fill=black] (-1,2) circle (1.8pt);
    \draw [fill=black] (-1,0) circle (1.8pt);
    \draw [fill=black] (2,2) circle (1.8pt);
    \draw [fill=black] (0.5,1) circle (1.8pt);
    \draw [fill=black] (2,0) circle (1.8pt);
    \end{scriptsize}
    \end{tikzpicture}
    \caption*{\large $J(\Gamma)$}
  \end{minipage}
  \begin{minipage}[b]{0.3\textwidth}\centering
    \begin{tikzpicture}[line cap=round,line join=round,>=triangle 45,x=0.6cm,y=0.6cm]
    \draw [line width=1pt] (-1,2)-- (-1,1);
    \draw [line width=1pt] (-1,1)-- (-1,0);
    \draw [line width=1pt] (1,2)-- (1,1);
    \draw [line width=1pt] (1,1)-- (1,0);
    \draw [line width=1pt] (-1,1)-- (1,1);
    \begin{scriptsize}
    \draw [fill=black] (-1,2) circle (1.8pt);
    \draw [fill=black] (-1,1) circle (1.8pt);
    \draw [fill=black] (-1,0) circle (1.8pt);
    \draw [fill=black] (1,2) circle (1.8pt);
    \draw [fill=black] (1,1) circle (1.8pt);
    \draw [fill=black] (1,0) circle (1.8pt);
    \draw (-1.4,1) node {a};
    \draw (-1,2.4) node {d};
    \draw (-1,-0.4) node {e};
    \draw (1.4,1) node {f};
    \draw (1,2.4) node {b};
    \draw (1,-0.4) node {c};
    \end{scriptsize}
    \end{tikzpicture}
    \caption*{\large $J^2(\Gamma)$}
  \end{minipage}
    \caption{A graph $\Gamma$ and its first and second  jump graphs.}\label{fig:first_examples}
\end{figure}

This work builds on content from ``Subgraph distances in graphs defined by edge transfers'' by Chartrand et al. \cite{Char}. The authors of this paper also consider sequences of graphs $\{J^k(G)\}$ generated by iterating the jump graph operation. A graph sequence $\{G_k\}$ \emph{converges} if there is some index $K$ and graph $G$ such that $G_k \cong G$ for all $k \geq K$. A sequence \emph{terminates} if it is finite and a sequence \emph{diverges} if it does not converge or terminate. Theorem 4 of \cite{Char} determines which graphs have iterated jump graph sequences that converge; this follows from work by Aigner in \cite{Aig}. The authors also classify all graphs which have a terminating iterated jump graph sequence in Theorem~7. To do this, they use subdivisions and vertex splittings of graphs to state and prove Lemma 5. These tools are analogous to our \emph{snipped subgraph} in Definition~\ref{defn:snipped} and Lemma~\ref{Lemma: H_snipped_of_G, J(H)_in_J(G)}. We extend the investigation in \cite{Char} by asking about the quality of diverging sequences of iterated jump graphs. Are there repeating sequences of jump graphs? Can the graphs in $\{J^k(G)\}$ stay ``small'' as $k \to \infty$? Our results stated in Theorem~\ref{Theorem: accumulating}, Theorem~\ref{Theorem: Exploding graph}, Theorem~\ref{Theorem: J^k(G)=G} give an answer to these questions and do not appear in \cite{Char}.

In Section~\ref{Section 2}, we define iterated jump graphs, $d$-value, and snipped subgraphs. We then give some immediate results about these objects. Section~\ref{Section 3} lists all graphs with terminating jump graph sequences, making a ``tree'' of graphs related by the jump graph operation (Figure~\ref{fig:dissipating_graphs} on page \pageref{fig:dissipating_graphs}). Graphs with non-terminating jump graph sequences are covered in Section~\ref{Section 4}. We present $C_5$ and the net graph as fundamental to our study (see Figure~\ref{fig:C5_and_N} on page \pageref{fig:C5_and_N}), introduce some useful tools for casework, and examine all connected graphs based on diameter. In Section~\ref{Section 5}, our work culminates in some truly fascinating results. In particular, we see that there is no graph except for $C_5$ or the net graph which gives itself for \emph{some} iterated jump graph. That is, no diverging sequence of iterated jump graphs repeats itself. We also show that for all diverging sequences $\{J^k(G)\}$, the number of edges in $J^k(G)$ grows without bound.

\section{Preliminaries}\label{Section 2}

A graph $G$ is \emph{simple} if any two vertices are connected by at most one edge and there are no edges from a single vertex to itself. An \emph{isolated vertex} is not the endpoint of any edge. Throughout the paper, we assume that all graphs are simple, finite, and nonempty unless explicitly stated. Prior to performing the jump graph operation, we will often consider two graphs equivalent if they differ by only isolated vertices since these have no effect on $J(G)$. We denote the vertex set of a graph $G$ by $V(G)$ and the number of vertices by $|V(G)|$; we denote the edge set by $E(G)$ and the number of edges by $|E(G)|$. For all terms not defined here, see a standard graph theory textbook such as \cite{West}.

\begin{defn}\label{mypick}
The \emph{jump graph} $J(G)$ of a graph $G$ has vertices given by the edges of $G$ (i.e. $V(J(G)) = E(G)$). For two vertices $u$ and $v$ of $J(G)$, the edge $\{u, v\}$ is in $E(J(G))$ if and only if edges $u$ and $v$ are not incident in $G$.
\end{defn}

\begin{figure}[H]
  \centering
    \begin{tikzpicture}[line cap=round,line join=round,>=triangle 45,x=0.8cm,y=0.8cm, scale=1]
    \draw [line width=1pt] (-7,0)-- (-5,0);
    \draw [line width=1pt] (-5,0)-- (-5,-2);
    \draw [line width=1pt] (-5,-2)-- (-7,-2);
    \draw [line width=1pt] (-7,-2)-- (-7,0);
    \draw [line width=1pt] (-4,0)-- (-4,-2);
    \draw (-5.93,-3) node {\large $\Gamma$};
    \draw [line width=1pt,color=black] (-1,0)-- (-1,-2);
    \draw [line width=1pt,color=black] (-1,-2)-- (0,-1);
    \draw [line width=1pt,color=black] (0,-1)-- (-1,0);
    \draw [line width=1pt,color=black] (0,-1)-- (1,0);
    \draw [line width=1pt,color=black] (1,0)-- (1,-2);
    \draw [line width=1pt,color=black] (1,-2)-- (0,-1);
    \draw (0.02, -3) node {\large $J(\Gamma)$};
    \begin{scriptsize}
    \draw [fill=black] (-7,0) circle (1.8pt);
    \draw [fill=black] (-5,0) circle (1.8pt);
    \draw[color=black] (-5.93,-0.2) node {$1$};
    \draw [fill=black] (-5,-2) circle (1.8pt);
    \draw[color=black] (-5.19,-0.88) node {$2$};
    \draw [fill=black] (-7,-2) circle (1.8pt);
    \draw[color=black] (-5.93,-1.76) node {$3$};
    \draw[color=black] (-6.85,-0.88) node {$4$};
    \draw [fill=black] (-4,0) circle (1.8pt);
    \draw [fill=black] (-4,-2) circle (1.8pt);
    \draw[color=black] (-4.25,-0.88) node {$5$};
    \draw [fill=black] (-1,0) circle (1.8pt);
    \draw[color=black] (-1.2,0.29) node {$1$};
    \draw [fill=black] (-1,-2) circle (1.8pt);
    \draw[color=black] (-1.2,-2.30) node {$3$};
    \draw [fill=black] (0,-1) circle (1.8pt);
    \draw[color=black] (0.02,-0.61) node {$5$};
    \draw [fill=black] (1,0) circle (1.8pt);
    \draw[color=black] (1.25,0.29) node {$2$};
    \draw [fill=black] (1,-2) circle (1.8pt);
    \draw[color=black] (1.25,-2.30) node {$4$};
    \end{scriptsize}
    \end{tikzpicture}
    \caption{A graph $\Gamma$ and its jump graph $J(\Gamma)$. Vertices in $J(\Gamma)$ are connected iff the corresponding edges in $\Gamma$ do \emph{not} share an endpoint.}\label{fig:jump_ex}
\end{figure}
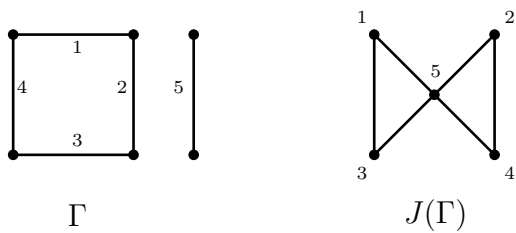

\begin{rem}
This notion is closely related to that of the matching complex; see, e.g., \cite[Chapter 11]{Jon} and \cite{Wachs}. Given a graph $G$, a \emph{matching} is a set of edges such that no two edges in the set are incident. The \emph{matching complex} of $G$ is the set of all matchings in $G$. Note that the jump graph contains only the matchings of cardinality at most two. In this way, the jump graph is a particular subset of the matching complex, known as its $1$-skeleton. Furthermore, a set of edges in $G$ forms a matching if and only if the corresponding vertices in $J(G)$ form a clique, a subset of vertices such that all two have an edge between them. Thus the jump graph encodes all of the same information as the matching complex. We came across the ideas for this paper through studying matching complexes; the project was rather far along when we discovered \cite{Char} and the study of jump graphs. Despite the direct connection between matching complexes and jump graphs, we do not know of any other sources relating the two. We wonder if there are results from the study of matching complexes which could be applied to jump graphs, or vice versa.
\end{rem}

Consider the connection between the jump graph and the well-studied \emph{line graph $L(G)$}. Vertices of the line graph represent edges in $G$, and these vertices are connected with an edge if and only if the corresponding edges are incident. With this definition, we see that $J(G)$ is the the complement of $L(G)$. The line graph has been an important tool for approaching our study of jump graphs. We will use this observation in Section~\ref{Section 4}.

With the above definition of the jump graph, there are a few immediate results we can state about subgraphs and induced subgraphs. A subgraph $H$ of $G$ is \emph{induced} if there are no two vertices in $H$ which are connected in $G$ but not in $H$.

\begin{lem} \label{Lemma: J(H)_subgraph_of_J(G)}
If $H$ is a subgraph of $G$, then $J(H)$ is an induced subgraph of $J(G)$.
\end{lem}

\begin{proof} Let $H$ be a subgraph of $G$. Every edge in $H$ is also an edge in $G$ so $V(J(H)) \subseteq V(J(G))$. If there are two edges $e_1,e_2$ in $H$ which are non-incident in $H$ then they are also non-incident in $G$, and so $E(J(H)) \subseteq E(J(G))$. Together these imply $J(H) \subseteq J(G)$. Furthermore if there are two edges $e_1$, $e_2$ in $H$ which are non-incident in $G$, then they must be non-incident in $H$. Thus $J(H)$ must be an induced subgraph of $J(G)$.
\end{proof}

Lemma~\ref{Lemma: J(H)_subgraph_of_J(G)} allows us to carry subgraphs forward under the jump graph operation. To go in the backwards direction, given some graph $G$, we want to identify a graph $G^*$ such that $J(G^*) \cong G$. After doing small examples, one observes that $C_3$ (a cycle on three vertices) and $S_3$ (a star with three pendants) have the same jump graph: three isolated vertices. This shows that there is not always a unique choice for $G^*$. Because of the connection between jump graphs and line graphs, we can apply Whitney's Graph Isomorphism Theorem which says, in effect, that $C_3$ and $S_3$ are the only examples of a graph with a non-unique line graph (and therefore jump graph as well).

\begin{lem}\label{Lemma: H_subgraph_of_G}
If $H$ and $G$ are connected graphs, $J(G) \neq J(C_3)$, $J(H) \neq J(C_3)$, and $J(H)$ is an induced subgraph of $J(G)$ then $H$ is a subgraph of $G$.
\end{lem}

\begin{proof}
Suppose $J(H)$ is an induced subgraph of $J(G)$. Then, since $J(G)^c = L(G)$, we know that $L(H)$ is an induced subgraph of $L(G)$. Because we assume that $J(H) \neq J(C_3)$, we know that $L(H) \neq L(C_3)$. Hence, by Whitney's Graph Isomorphism Theorem \cite{Whit}, $H$ is uniquely determined. By the same logic, since $J(G) \neq J(C_3)$ then $G$ is uniquely determined as well. Then $H$ must be a subgraph of $G$.
\end{proof}

Notice that Lemma~\ref{Lemma: H_subgraph_of_G} is a partial converse of Lemma~\ref{Lemma: J(H)_subgraph_of_J(G)}. We can expand Lemma~\ref{Lemma: H_subgraph_of_G} to include some disconnected graphs with more conditions, but this is not relevant to the bulk of the paper. This ``backwards jump graph'' operation is used only in Section~\ref{Section 3} as we construct Figure~\ref{fig:dissipating_graphs}. Now, we turn attention towards the main object of study: \emph{iterated jump graphs}.

\begin{defn}
Let $G$ be a graph and define $J^0(G) = G$. For $k \geq 1$, the \emph{$k^{\textrm{th}}$ jump graph}, denoted $J^k(G)$, is the jump graph of $J^{k-1}(G)$.
\end{defn}

Our goal is to study the behavior of the sequence $\{J^k(G)\}$. Of particular interest is determining whether a given graph dissipates in the following sense.

\begin{defn}
A graph $G$ \emph{dissipates} if there is some $k \geq 0$ such that $J^k(G) = \emptyset$. The \emph{dissipation number}, denoted $d(G)$, is the smallest $k \geq 0$ such that $J^k(G) = \emptyset$. If there is no such $k$, then $d(G) = \infty$.
\end{defn}

\begin{rem}
A graph $G$ dissipates if and only if the sequence $\{J^k(G)\}$ terminates. The dissipation number of $G$ is the length of the sequence $\{J^k(G)\}$ where we do not include $\emptyset$.
\end{rem}

For examples of dissipation, see Figure~\ref{fig:dissipating_graphs} on page \pageref{fig:dissipating_graphs}. The reader can find the graph $\Gamma$ from Figure~\ref{fig:jump_ex} in the lower left-hand corner. By counting the arrows between $\Gamma$ and the empty set, we see that $d(\Gamma) = 7$.

If $d(G) < \infty$, we will sometimes say that $G$ is $d$-finite. Otherwise, we will say that $G$ is $d$-infinite. In studying the behavior of $J^k(G)$, we will see that the following two notions play an integral part. 

\begin{defn}
Suppose we have a graph $G$ with vertex set $V(G) = \{v_1,\dots,v_n$\} and edge set $E(G)$. A \emph{quotient graph} $Q$ of $G$ is defined in the following way. Take some partition of $V(G)$ and then apply the equivalence relation created by this partition, letting $[v_i]$ denote such an equivalence class. The vertex set and edge set of $Q$ are defined below.
\[V(Q) = \Big\{[v_i] : v_i \in V(G)\Big\} \hspace{30pt} E(Q) = \Big\{\{[v_i],[v_j]\} : \{v_i,v_j\} \in E(G), [v_i] \neq [v_j]\Big\}\]
\end{defn}

Note that, by construction, $Q$ will not have any double edges or loops. In a qualitative way, a quotient graph is obtained by gluing vertices together and then deleting double edges and loops.

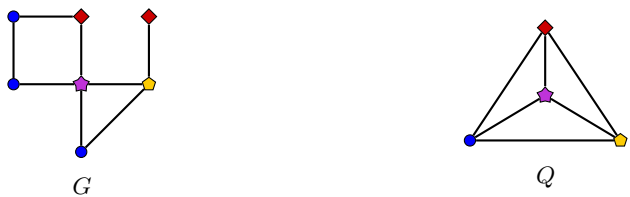
\begin{figure}[H]
    \centering
    \begin{subfigure}{0.4\textwidth}
        \centering
        \begin{tikzpicture}[scale=0.45]
        \node (a) [vertex,draw=black,fill=blue] at (-2,2) {};
        \node (b) [vertex,diamond,draw=black,fill=red] at (0,2) {};
        \node (c) [vertex,star,star points=5,draw=black,fill=pur] at (0,0) {};
        \node (d) [vertex,draw=black,fill=blue] at (-2,0) {};
        \node (e) [vertex,draw=black,fill=blue] at (0,-2) {};
        \node (f) [vertex,regular polygon,regular polygon sides=5,draw=black,fill=gold] at (2,0) {};
        \node (g) [vertex,diamond,draw=black,fill=red] at (2,2) {};
        
        \draw [edge] (a) -- (b) -- (c) -- (d) -- (a);
        \draw [edge] (c) -- (e) -- (f) -- (c);
        \draw [edge] (f) -- (g);
        \end{tikzpicture}
        \caption*{$G$}\label{fig:short_path}
    \end{subfigure}
    \begin{subfigure}{0.4\textwidth}
        \centering
        \begin{tikzpicture}
            \node (a) [vertex,diamond,draw=black,fill=red] at (0,0.9) {};
            \node (b) [vertex,regular polygon,regular polygon sides=5,draw=black,fill=gold] at (1,-0.6) {};
            \node (c) [vertex,star,star points=5,draw=black,fill=pur] at (0,0) {};
            \node (d) [vertex,draw=black,fill=blue] at (-1,-0.6) {};
            
            \draw [edge] (a) -- (b) -- (c) -- (a);
            \draw [edge] (b) -- (d) -- (c);
            \draw [edge] (a) -- (d);
        \end{tikzpicture}
        \caption*{$Q$}
    \end{subfigure}
    \caption{A graph $G$ and a quotient $Q$ of $G$. The different colors and shapes of vertices represent the partition of $V(G)$.} \label{fig:quotient_example}
\end{figure}

\begin{defn}\label{defn:snipped}
A \emph{snipped subgraph} of a graph $G$ is a quotient graph of a subgraph of $G$.
\end{defn}

\begin{example}
Figure~\ref{fig:snipped_example} shows an example of a graph $G$ and three of its snipped subgraphs. The subgraph of $G$ of which $H_i$ is a quotient can be identified by the edge labels. Observe that $H_1$ is a subgraph of $G$ but $H_2$ and $H_3$ are nontrivial quotients of a subgraph of $G$.
\end{example}

\begin{rem}
In particular, any subgraph of $G$ is also a snipped subgraph of $G$. This means $G$ is a snipped subgraph of itself.
\end{rem}

We have defined a snipped subgraph by starting with $G$, taking a subgraph, and gluing together some of the vertices. This direction is helpful for understanding the definition, but the motivation is more clear if we conceptualize the snipped subgraph in another way. Suppose $H$ is a snipped subgraph of $G$. Split apart the quotiented vertices of $H$ and overlay it on top of $G$. Notice that the vertex-splitting action preserves disconnections among edges in $H$. Looking back at the example in Figure~\ref{fig:snipped_example}, we notice that in $H_1$, edges 1 and 6 are non-incident and they are still non-incident in $G$. However, in graph $H_2$, edges 1 and 6 are incident, but these edges are non-incident in $G$.

Edges in $J(H)$ correspond to disconnections between edges in $H$. In this way, we know edges in $J(H)$ will be preserved under ``snipping'' of $H$. This property allows us to state the result in Lemma~\ref{Lemma: H_snipped_of_G, J(H)_in_J(G)}. 

\begin{lem}\label{Lemma: baby}
Suppose that $H$ is a quotient graph of $G$. Then there exists some subgraph $G' \subseteq G$ such that $|E(G')| = |E(H)|$ and $H$ is a quotient graph of $G'$
\end{lem}

The proof of this lemma amounts to choosing an edge $\{v, u\}$ in $G$ for each edge $\{[v], [u]\}$ in $H$ such that, as the notation suggests, $\{v, u\}$ is mapped to $\{[v], [u]\}$ under the quotient identifications.


\begin{figure}
  \centering
  \begin{minipage}[b]{0.3\textwidth}\centering
    \begin{tikzpicture}[line cap=round,line join=round,>=triangle 45,x=0.8cm,y=0.8cm]
    \draw [line width=1pt] (-3,3.5)-- (-2,2);
    \draw [line width=1pt] (-2,2)-- (-3,0.45);
    \draw [line width=1pt] (-3.5,2)-- (-2,2);
    \draw [line width=1pt] (-2,2)-- (-1,2);
    \draw [line width=1pt] (-1,2)-- (0,3);
    \draw [line width=1pt] (-1,2)-- (0,1);
    \draw [line width=1pt] (-3.5,2)-- (-3,0.45);
    \draw [line width=1pt] (0,3)-- (0,1);
    \begin{scriptsize}
    \draw [fill=black] (-3,3.5) circle (1.8pt);
    \draw [fill=black] (-2,2) circle (1.8pt);
    \draw[color=black] (-2.7,2.7) node {1};
    \draw [fill=black] (-3,0.45) circle (1.8pt);
    \draw[color=black] (-2.42,1.1) node {4};
    \draw [fill=black] (-3.5,2) circle (1.8pt);
    \draw[color=black] (-2.84,1.8) node {2};
    \draw [fill=black] (-1,2) circle (1.8pt);
    \draw[color=black] (-1.44,1.8) node {5};
    \draw [fill=black] (0,3) circle (1.8pt);
    \draw[color=black] (-0.75,2.6) node {6};
    \draw [fill=black] (0,1) circle (1.8pt);
    \draw[color=black] (-0.8,1.4) node {7};
    \draw[color=black] (-3.6,1.27) node {3};
    \draw[color=black] (0.28,2.15) node {8};
    \end{scriptsize}
    \end{tikzpicture}
    \caption*{\large $G$}
  \end{minipage} \\
  \begin{minipage}[b]{0.3\textwidth}\centering
    \begin{tikzpicture}[line cap=round,line join=round,>=triangle 45,x=0.7cm,y=0.7cm]
    \draw [line width=1pt] (-2,1)-- (-1,0);
    \draw [line width=1pt] (-3,0)-- (-1,0);
    \draw [line width=1pt] (0,0)-- (1,1);
    \draw [line width=1pt] (0,0)-- (1,-1);
    \draw [line width=1pt] (-2,-1)-- (-1,0);
    \begin{scriptsize}
    \draw [fill=black] (-2,1) circle (1.8pt);
    \draw [fill=black] (-1,0) circle (1.8pt);
    \draw[color=black] (-1.64,1.01) node {1};
    \draw [fill=black] (1,1) circle (1.8pt);
    \draw[color=black] (0.38,0.99) node {6};
    \draw [fill=black] (1,-1) circle (1.8pt);
    \draw [fill=black] (-3,0) circle (1.8pt);
    \draw[color=black] (-2.14,0.39) node {2};
    \draw [fill=black] (0,0) circle (1.8pt);
    \draw[color=black] (0.3,-0.65) node {7};
    \draw [fill=black] (-2,-1) circle (1.8pt);
    \draw[color=black] (-1.2,-0.49) node {4};
    \end{scriptsize}
    \end{tikzpicture}
    \caption*{\large $H_1$}
  \end{minipage}
  \begin{minipage}[b]{0.3\textwidth}\centering
    \begin{tikzpicture}[line cap=round,line join=round,>=triangle 45,x=0.7cm,y=0.7cm]
    \draw [line width=1pt] (5,1)-- (7,1);
    \draw [line width=1pt] (7,1)-- (7,-1);
    \draw [line width=1pt] (7,-1)-- (5,-1);
    \draw [line width=1pt] (5,-1)-- (5,1);
    \begin{scriptsize}
    \draw [fill=black] (5,1) circle (1.8pt);
    \draw [fill=black] (7,1) circle (1.8pt);
    \draw[color=black] (6.1,1.41) node {1};
    \draw [fill=black] (7,-1) circle (1.8pt);
    \draw[color=black] (7.2,0.11) node {6};
    \draw [fill=black] (5,-1) circle (1.8pt);
    \draw[color=black] (6.02,-1.23) node {7};
    \draw[color=black] (4.6,0.17) node {4};
    \end{scriptsize}
    \end{tikzpicture}
    \caption*{\large $H_2$}
  \end{minipage}
  \begin{minipage}[b]{0.3\textwidth}\centering
    \begin{tikzpicture}[line cap=round,line join=round,>=triangle 45,x=0.7cm,y=0.7cm]
    \draw [line width=1pt] (12,2)-- (13.5,1);
    \draw [line width=1pt] (13.5,1)-- (13,-1);
    \draw [line width=1pt] (13,-1)-- (11,-1);
    \draw [line width=1pt] (10.5,1)-- (12,2);
    \draw [line width=1pt] (10.5,1)-- (11,-1);
    \begin{scriptsize}
    \draw [fill=black] (12,2) circle (1.8pt);
    \draw [fill=black] (13.5,1) circle (1.8pt);
    \draw[color=black] (12.8,1.8) node {3};
    \draw [fill=black] (13,-1) circle (1.8pt);
    \draw[color=black] (13.5,0.19) node {8};
    \draw [fill=black] (11,-1) circle (1.8pt);
    \draw[color=black] (12.02,-1.23) node {7};
    \draw [fill=black] (10.5,1) circle (1.8pt);
    \draw[color=black] (10.96,1.83) node {1};
    \draw[color=black] (10.5,0.17) node {5};
    \end{scriptsize}
    \end{tikzpicture}
    \caption*{\large $H_3$}
  \end{minipage}
    \caption{Graphs $H_1$, $H_2$, and $H_3$ are snipped subgraphs of $G$}\label{fig:snipped_example}
\end{figure}

\begin{lem} \label{Lemma: H_snipped_of_G, J(H)_in_J(G)} If $H$ is a snipped subgraph of $G$, then $J(H)$ is a subgraph of $J(G)$. 
\end{lem}

\begin{proof}
Since $H$ is a quotient of a subgraph of $G$, there is some subgraph $G' \subseteq G$ such that $H$ is a quotient graph of $G'$ and $|E(G')| = |E(H)|$ by Lemma~\ref{Lemma: baby}. We will show that $J(H)$ is a subgraph of $J(G')$; then by Lemma~\ref{Lemma: H_subgraph_of_G}, $J(G') \subseteq J(G)$ so we have $J(H) \subseteq J(G)$. To do this, we will find a graph homomorphism from $J(H)$ to $J(G')$: an injective map $\varphi : V(J(H)) \to V(J(G'))$ such that 
\begin{equation}\label{eq:edge_condition}
\text{if } \{x, y\} \in E(J(H)), \text{ then } \{\varphi(x), \varphi(y)\} \in E(J(G')). 
\end{equation}

For a vertex $v \in V(G')$ let $[v]$ be its equivalence class given by the quotient, that is, its image under the quotient map from $G'$ to $H$. Let $q : E(G') \to E(H)$ be the quotient map such that $q(\{u, v\}) = \{[u], [v]\}$. This map is surjective by the definition of a quotient graph. An edge $\{[u], [v]\}$ exists in $E(H)$ if and only if there is some $u' \in [u]$ and $v' \in [v]$ such that $\{v', u'\} \in E(G')$. We also note that $E(G')$ and $E(H)$ are finite sets with the same cardinality so the map $q$ must also be injective. Hence, we can define an inverse function $\varphi := q^{-1} : E(H) \to E(G')$. For each edge $\{x, y\} \in E(H)$, there is a unique choice of $\{u, v\} \in E(G')$ such that $[u] = x$ and $[v] = y$.

Next, by definition of jump graphs, $E(H) = V(J(H))$ and $E(G') = V(J(G'))$. This means that $\varphi$ is a bijection from $V(J(H))$ to $V(J(G'))$.

Lastly, we show that $\varphi$ satisfies the edge condition (\ref{eq:edge_condition}). Let an edge $\{u, v\}$ in $G'$ be denoted by $uv$ and let an edge $\{[u], [v]\}$ in $H'$ be denoted by $[u][v]$. Suppose that $\{[u][v], [a][b]\}$ is an edge in $J(H)$. Then $[u][v]$ and $[a][b]$ are non-incident edges in $H$ so the endpoints $[u]$, $[v]$, $[a]$, and $[b]$ are all distinct in $H$. Without loss of generality, let
\[\varphi([u][v]) = uv \hspace{10pt} \text{and} \hspace{10pt} \varphi([a][b]) = ab.\]
The endpoints $u$, $v$, $a$, and $b$ must all be distinct in $G'$ since their images under $q$ are distinct in $H$. Thus, edges $uv$ and $ab$ are not incident in $G'$ and so $\{uv, ab\}$ is an edge of $J(G')$.

Hence, we see that $\varphi$ is an injective graph homomorphism and so $J(H)$ is a subgraph of $J(G')$. And since $J(G') \subseteq J(G)$ by Lemma~\ref{Lemma: H_subgraph_of_G}, we also conclude that $J(H) \subseteq J(G)$.
\end{proof}

Consider Lemmas~\ref{Lemma: J(H)_subgraph_of_J(G)} and \ref{Lemma: H_snipped_of_G, J(H)_in_J(G)} side by side. Lemma~\ref{Lemma: J(H)_subgraph_of_J(G)} implies the existence of a certain \emph{induced} subgraph of $J(G)$ while Lemma~\ref{Lemma: H_snipped_of_G, J(H)_in_J(G)} implies the existence of a certain subgraph of $J(G)$. Because every subgraph is also itself a \emph{snipped} subgraph, the set of subgraphs of $G$ is contained in the set of snipped subgraphs of $G$. Hence, it is easier to satisfy the assumptions of Lemma~\ref{Lemma: H_snipped_of_G, J(H)_in_J(G)} than those of Lemma~\ref{Lemma: J(H)_subgraph_of_J(G)}. Recognizing also that induced subgraphs are of no particular use to this problem, Lemma~\ref{Lemma: H_snipped_of_G, J(H)_in_J(G)} will be our primary tool as we continue in our discussion.

\begin{prop}\label{Prop: subgraph-d-value}
If $H$ is a snipped subgraph of $G$, then $d(H) \leq d(G)$.
\end{prop}

\begin{proof}
    First, consider if $d(G) = \infty$. Then, trivially we have $d(H) \leq d(G) = \infty$.
    
    Now, consider when $d(G)$ is finite and let $d(G) = d_0$. If $d_0 = 0$ then $G = \emptyset$ and the result is immediate because $\emptyset$ is the only snipped subgraph of $\emptyset$. Thus, let $d_0$ be at least $1$. Applying Lemma~\ref{Lemma: H_snipped_of_G, J(H)_in_J(G)} iteratively,
    \[J^k(H) \subseteq J^k(G) \]
    for all $k \geq 1$. Now, let $k = d_0$ and the above statement becomes
    \[J^{d_0}(H) \subseteq J^{d_0}(G) = \emptyset . \]
    
    \noindent And so we have $J^{d_0}(H) = \emptyset$ and it follows that
    
    \[d(H) \leq d_0 = d(G).\qedhere \]\end{proof}

\begin{cor}\label{Corollary: snipped_d=infinity}
If $H$ is a snipped subgraph of $G$ and $d(H) = \infty$, then $d(G) = \infty$.
\end{cor}

Corollary~\ref{Corollary: snipped_d=infinity} reveals our primary use of snipped subgraphs. When presented with a graph $G$, we can look for any snipped subgraph $H$ which we know to have $d(H) = \infty$. If we can find such a graph then we know that $d(G) = \infty$ as well. This makes it much simpler to determine if a given graph will dissipate or not. We will see Corollary~\ref{Corollary: snipped_d=infinity} become very relevant in Section~\ref{Section 4}.

\begin{rem}
In \cite{Char}, the authors use vertex splitting and subdivisions of graphs to serve the same purpose as our snipped subgraph. These tools are applied when they state and prove Lemma 5 which is equivalent to our Corollary~\ref{Corollary: snipped_d=infinity}. We found the concept of the snipped subgraph useful because it combines both vertex splitting and graph subdivisions into one tool, getting at the crux of when we see $J(H)$ appear as a subgraph of $J(G)$.
\end{rem}

\section{Graphs with finite $d$ value}\label{Section 3}

There are relatively few graphs with a finite $d$ value. If $G$ is $d$-finite then $J^k(G)$ is also $d$-finite for all $k$ and $J^{d(G)}(G) = \emptyset$. Using these observations, we can build a tree of all $d$-finite graphs by ``working backwards'' from the empty set. To start, $J^k(G) = \emptyset$ if and only if $J^{k - 1}(G)$ is the empty set or a set of isolated vertices. Considering this latter possibility, we can find all the options for $J^{k - 2}(G)$. We continue doing this ``backwards jump graph operation'' until we are left with some $J^{k - n}(G)$ which cannot be the jump graph of any graph. To know when this point is reached, we use Beineke's characterization of line graphs from ``Characterizations of derived graphs'' \cite{Bein}.


In this 1970 paper, Beineke showed that a graph is a line graph for another graph if and only if it does not have one of the nine induced subgraphs in Figure~\ref{fig:beineke_line_graphs} \cite{Bein}. Recall that the line graph is the complement of the jump graph. This implies that if some graph $G$ is the jump graph of another graph $G^*$ (i.e. $G = J(G^*)$) then $G^c$ must not have any of Beineke's nine forbidden induced subgraphs. Equivalently, $G$ must not have the complement of any forbidden graph as an induced subgraph. These complements are shown in Figure~\ref{fig:beineke_jump_graphs}. 

The tree of graphs made through this process is shown in Figure~\ref{fig:dissipating_graphs}. Infinite families of graphs are blocked in grey and isolated vertices are not explicitly drawn where they may appear. For example, see that the graphs $C_4$ and $C_4$ with a diagonal edge (in the lower third of the figure) lead to the same graph in Figure~\ref{fig:dissipating_graphs}. We ignore isolated vertices because they have no effect on $J(G)$. Induced subgraphs from Figure~\ref{fig:beineke_jump_graphs} are highlighted in red.

While we only outline an idea for a proof here, Figure~\ref{fig:dissipating_graphs} does contains all $d$-finite graphs as shown rigorously in \cite{Char}. The reader can look at Theorem 7 in \cite{Char} if they would like to see the list in Figure~\ref{fig:dissipating_graphs} catalogued in rigor. 

\begin{figure}
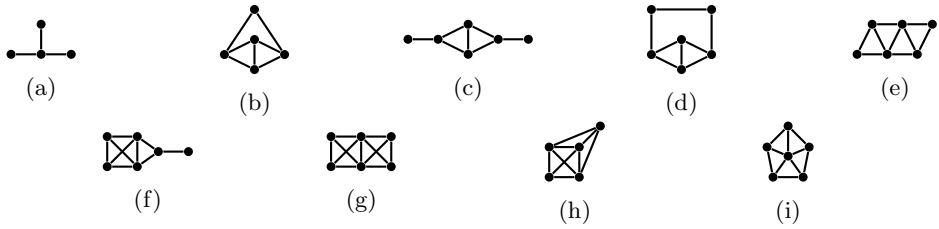

\centering
\begin{subfigure}{0.18\textwidth}
\centering
\mygraph{1[x=0,y=0] -- 2[x=1, y=0] -- 3[x=2,y=0], 2 -- 4[x=1,y=1]}
\caption{}
\end{subfigure} \begin{subfigure}{0.18\textwidth}
\centering
\mygraph{1[x=0,y=0.5] -- 2[x=1, y=0] -- 3[x=1,y=1] -- 1, 2 -- 4[x=2,y=0.5] -- 3, 1 -- 5[x=1,y=2] -- 4}
\caption{}
\end{subfigure} \begin{subfigure}{0.18\textwidth}
\centering
\mygraph{1[x=0,y=0.5] -- 2[x=1, y=0] -- 3[x=1,y=1] -- 1, 2 -- 4[x=2,y=0.5] -- 3, 1 -- 5[x=-1,y=0.5], 4 -- 6[x=3,y=0.5]}
\caption{}
\end{subfigure} \begin{subfigure}{0.18\textwidth}
\centering
\mygraph{1[x=0,y=0.5] -- 2[x=1, y=0] -- 3[x=1,y=1] -- 1, 2 -- 4[x=2,y=0.5] -- 3, 1 -- 5[x=0,y=2] -- 6[x=2,y=2] -- 4}
\caption{}
\end{subfigure} \begin{subfigure}{0.18\textwidth}
\centering
\mygraph{1[x=0,y=0] -- 2[x=1,y=0] -- 3[x=2,y=0], 4[x=0.5,y=1] -- 5[x=1.5,y=1] -- 6[x=2.5,y=1], 1 -- 4 -- 2 -- 5 -- 3 -- 6}
\caption{}
\end{subfigure} \begin{subfigure}{0.18\textwidth}
\centering
\mygraph{1[x=0,y=0] -- 2[x=1,y=0] -- 3[x=1,y=1] -- 4[x=0,y=1] -- 1 -- 3, 2 -- 4, 2 -- 5[x=1.7,y=0.5] -- 3, 5 -- 6[x=2.7,y=0.5]}
\caption{}
\end{subfigure} \begin{subfigure}{0.18\textwidth}
\centering
\mygraph{1[x=0,y=0] -- 2[x=1,y=0] -- 3[x=2,y=0] -- 4[x=2,y=1] -- 5[x=1,y=1] -- 6[x=0,y=1] -- 1, 1 -- 5 -- 3, 4 -- 2 -- 6, 2 -- 5}
\caption{}
\end{subfigure} \begin{subfigure}{0.18\textwidth}
\centering
\mygraph{1[x=0,y=0] -- 2[x=1,y=0] -- 3[x=1,y=1] -- 4[x=0,y=1] -- 1 -- 3, 2 -- 4, 2 -- 5[x=1.7,y=1.7] -- 4, 3 -- 5}
\caption{}
\end{subfigure} \begin{subfigure}{0.18\textwidth}
\centering
\mygraph{1[x=0,y=0] -- 2[x=1,y=0] -- 3[x=1.2,y=1] -- 4[x=0.5,y=1.7] -- 5[x=-0.2,y=1] -- 1, 6[x=0.5,y=0.7] -- {1,2,3,4,5}}
\caption{}
\end{subfigure}
\caption{The nine graph which Beineke's theorem says cannot be an induced subgraph of any line graph.} \label{fig:beineke_line_graphs}
\end{figure}

\begin{figure}
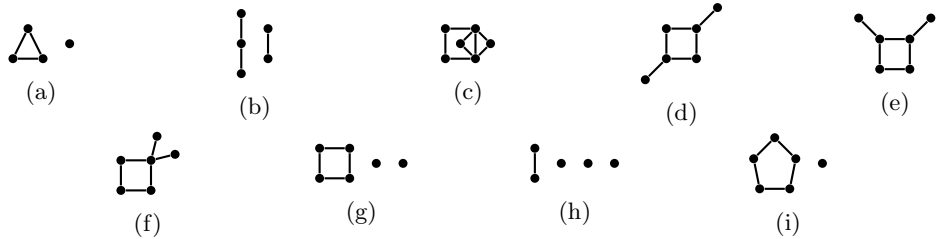

\centering
\begin{subfigure}{0.18\textwidth}
\centering
\mygraph{1[x=0,y=0] -- 2[x=1, y=0] -- 3[x=0.5,y=1] -- 1,} \ \mygraph{1}
\caption{}
\end{subfigure} \begin{subfigure}{0.18\textwidth}
\centering
\mygraph{1[x=0,y=0] -- 2[x=0,y=1] -- 3[x=0,y=2]} \ \mygraph{1[x=0,y=0] -- 2[x=0,y=1]}
\caption{}
\end{subfigure} \begin{subfigure}{0.18\textwidth}
\centering
\mygraph{1[x=0,y=0] -- 2[x=1,y=0] -- 3[x=1,y=1] -- 4[x=0,y=1] -- 1, 2 -- 5[x=0.5,y=0.5] -- 3, 2 -- 6[x=1.5,y=0.5] -- 3}
\caption{}
\end{subfigure} \begin{subfigure}{0.18\textwidth}
\centering
\mygraph{1[x=0,y=0] -- 2[x=1,y=0] -- 3[x=1,y=1] -- 4[x=0,y=1] -- 1, 1 -- 5[x=-0.7,y=-0.7], 3 -- 6[x=1.7,y=1.7]}
\caption{}
\end{subfigure} \begin{subfigure}{0.18\textwidth}
\centering
\mygraph{1[x=0,y=0] -- 2[x=1,y=0] -- 3[x=1,y=1] -- 4[x=0,y=1] -- 1, 4 -- 5[x=-0.7,y=1.7], 3 -- 6[x=1.7,y=1.7]}
\caption{}
\end{subfigure} \begin{subfigure}{0.18\textwidth}
\centering
\mygraph{1[x=0,y=0] -- 2[x=1,y=0] -- 3[x=1,y=1] -- 4[x=0,y=1] -- 1, 3 -- 5[x=1.8,y=1.2], 3 -- 6[x=1.2,y=1.8]}
\caption{}
\end{subfigure} \begin{subfigure}{0.18\textwidth}
\centering
\mygraph{1[x=0,y=0] -- 2[x=1,y=0] -- 3[x=1,y=1] -- 4[x=0,y=1] -- 1} \ \mygraph{1} \ \mygraph{1}
\caption{}
\end{subfigure} \begin{subfigure}{0.18\textwidth}
\centering
\mygraph{1[x=0,y=0] -- 2[x=0,y=1]} \ \mygraph{1} \ \mygraph{1} \ \mygraph{1}
\caption{}
\end{subfigure} \begin{subfigure}{0.18\textwidth}
\centering
\mygraph{1[x=0,y=0] -- 2[x=1,y=0] -- 3[x=1.2,y=1] -- 4[x=0.5,y=1.7] -- 5[x=-0.2,y=1] -- 1} \ \mygraph{1}
\caption{}
\end{subfigure}
\caption{Complements of the nine graphs from Figure~\ref{fig:beineke_line_graphs}; equivalently, the graphs which cannot occur as induced subgraphs of any jump graph.} \label{fig:beineke_jump_graphs}
\end{figure}

\begin{figure}
\centering
\scalebox{0.92}{
\begin{tikzpicture}
\graph [grow left=2.23cm, branch down sep=0.4cm, left anchor=west, right anchor=east, edges={out=180, in=0, max distance=1cm}] {
    empty / $\large \emptyset$ <-
    vertices / \mygraph{1[x=0,y=-2], 2[x=0,y=-1], 3[x=0,y=0], 4[x=0,y=-3.5, rectangle, inner sep=0pt, fill=none, typeset=\smash{\footnotesize\vdots}]} <- {
        K3 / \mygraph{1[x=0,y=0] -- 2[x=1,y=0] -- 3[x=0.5,y=1] -- 1} <- 
            K2 x 3 / \mygraph{1[x=0,y=0] -- 2[x=0,y=1]} \ \mygraph{1[x=0,y=0] -- 2[x=0,y=1]} \ \mygraph{1[x=0,y=0] -- 2[x=0,y=1]} <-
                K4 / \mygraph{1[x=0,y=0] -- 2[x=1,y=0] -- 3[x=0,y=1] -- 4[x=1,y=1] -- 1 -- 3, 2 -- 4} <-
                    K2 x 4 / \mygraph {1[x=0,y=0, red] --[red] 2[x=0,y=1, red]} \ \mygraph {1[x=0,y=0, red] -- 2[x=0,y=1]} \ \mygraph {1[x=0,y=0, red] -- 2[x=0,y=1]} \ \mygraph {1[x=0,y=0, red] -- 2[x=0,y=1]},
    Sn / $S_n$ -!- {
        K3 one leaf / \mygraph{1[x=0,y=0] -- 2[x=0,y=1] -- 3[x=1,y=0.5] -- 1, 3 -- 4[x=2,y=0.5]} <- 
            P4 + P2 / \mygraph{1[x=0,y=0, red] --[red] 2[x=1,y=0, red] --[red] 3[x=2,y=0, red] -- 4[x=3,y=0]} \ \mygraph{1[x=0,y=0, red] --[red] 2[x=0,y=1, red]},
        K3 two leaves / \mygraph{1[x=0,y=0] -- 2[x=0,y=1] -- 3[x=1,y=0.5] -- 1, 3 -- 4[x=2,y=0], 3 -- 5[x=2,y=1]} <-
            K3 one leaf + P2 / \mygraph{1[x=0,y=0, red] -- 2[x=0,y=1] -- 3[x=1,y=0.5, red] --[red] 1, 3 --[red] 4[x=2,y=0.5, red]} \ \mygraph{1[x=0,y=0, red] --[red] 2[x=0,y=1, red]},
        K3 n leaves / \mygraph{1[x=0,y=0, red] --[red] 2[x=0,y=1, red] -- 3[x=1,y=0.5] -- 1, 3 -- 4[x=2,y=1, red], 3 -- 5[x=2,y=0.66, red], 3 -- 6[x=2,y=0.33, red], 7[x=2,y=-1, rectangle, inner sep=0pt, fill=none, typeset=\smash{\footnotesize\vdots}]},
        P4 / \mygraph{1[x=0,y=0] -- 2[x=1,y=0] -- 3[x=2,y=0] -- 4[x=3,y=0]} <- {
            P5 / \mygraph{1[x=0,y=0] -- 2[x=1,y=0] -- 3[x=2,y=0] -- 4[x=3,y=0] -- 5[x=4,y=0]} <- {
                C4 leaf / \mygraph{1[x=0,y=0] -- 2[x=1,y=0] -- 3[x=1,y=1] -- 4[x=0,y=1] -- 1, 3 -- 5[x=1.7,y=1.7]} <- {
                    weird graph 1 / \mygraph{1[x=0,y=0, red] --[red] 2[x=1,y=0, red] -- 3[x=2,y=0] -- 4[x=3,y=0, red] --[red] 5[x=4,y=0, red], 2 --[red] 6[x=1,y=1, red]},
                    weird graph 2 / \mygraph{1[x=0,y=0] -- 2[x=1,y=0] -- 3[x=2,y=0] -- 4[x=3,y=0], 5[x=1,y=1] -- 2 -- 6[x=2,y=1] -- 3} <-
                        weird graph 3 / \mygraph{1[x=0,y=0, red] --[red] 2[x=0,y=1, red] --[red] 3[x=1,y=0.5, red] --[red] 1, 3 -- 4[x=2,y=0] -- 5[x=3,y=0, red], 3 -- 6[x=2,y=1]}},
                diag square leaf / \mygraph{1[x=0,y=0] -- 2[x=1,y=0] -- 3[x=1,y=1] -- 4[x=0,y=1] -- 1, 1 -- 3, 3 -- 5[x=1.7,y=1.7]} <-
                    weird graph 4 / \mygraph{1[x=0,y=0, red] --[red] 2[x=1,y=0, red] -- 3[x=2,y=0] -- 4[x=3,y=0], 2 --[red] 5[x=1,y=1, red]} \ \mygraph{1[x=0,y=0, red] --[red] 2[x=0,y=1, red]}},
            almost pendant / \mygraph{1[x=0,y=0] -- 2[x=1,y=0] -- 3[x=2,y=0] -- 4[x=3,y=0], 2 -- 5[x=1.5,y=1] -- 3} <-
                E shape / \mygraph{1[x=0,y=0] -- 2[x=1,y=0] -- 3[x=2,y=0] -- 4[x=3,y=0] -- 5[x=4,y=0], 3 -- 6[x=2,y=1]} <-
                    other diag square leaf / \mygraph{1[x=0,y=0, red] --[red] 2[x=1,y=0, red] -- 3[x=1,y=1] -- 4[x=0,y=1, red] --[red] 1, 2 --[red] 4, 3 -- 5[x=1.7,y=1.7, red]}},
        P3 two leaves / \mygraph{1[x=0,y=0] -- 2[x=1,y=0] -- 3[x=2,y=0] -- 4[x=3,y=0.5], 3 -- 5[x=3,y=-0.5]} <-
            K3 length 2 tail / \mygraph{1[x=0,y=0, red] --[red] 2[x=0,y=1, red] --[red] 3[x=1,y=0.5, red] --[red] 1, 3 -- 4[x=2,y=0.5] -- 5[x=3,y=0.5, red]},
        P3 n leaves / \mygraph{1[x=0,y=0.5, red] --[red] 2[x=1,y=0.5, red] -- 3[x=2,y=0.5], 3 -- 4[x=3,y=1, red], 3 -- 5[x=3,y=0.66, red], 3 -- 6[x=3,y=0.33, red], 7[x=3,y=-1, rectangle, inner sep=0pt, fill=none, typeset=\smash{\footnotesize\vdots}]},
        K2 + K2 / \mygraph{1[x=0,y=0] -- 2[x=0,y=1]} \ \mygraph{1[x=0,y=0.5] -- 2[x=1,y=0.5]} <- {
            diag square / \mygraph{1[x=0,y=0] -- 2[x=1,y=0] -- 3[x=1,y=1] -- 4[x=0,y=1] -- 1, 1 -- 3} <-
                P3 + 2 x P2 / \mygraph{1[x=0,y=0, red] --[red] 2[x=0,y=1, red] --[red] 3[x=0,y=2, red]} \ \mygraph{1[x=0,y=0, red] --[red] 2[x=0,y=1, red]} \ \mygraph{1[x=0,y=0] -- 2[x=0,y=1]},
            C4 / \mygraph{1[x=0,y=0] -- 2[x=1,y=0] -- 3[x=1,y=1] -- 4[x=0,y=1] -- 1} <- {
                2 x P3 / \mygraph{1[x=0,y=0, red] --[red] 2[x=0,y=1, red] --[red] 3[x=0,y=2, red]} \ \mygraph{1[x=0,y=0, red] --[red] 2[x=0,y=1, red] -- 3[x=0,y=2]},
                H shape / \mygraph{1[x=0,y=0] -- 2[x=1,y=0] -- 3[x=2,y=0] -- 4[x=3,y=0], 2 -- 5[x=1,y=1], 3 -- 6[x=2,y=1]} <-
                    bow tie / \mygraph{1[x=0,y=0] -- 2[x=0,y=1] -- 3[x=1,y=0.5] -- 1, 3 -- 4[x=2,y=0] -- 5[x=2,y=1] -- 3} <-
                        C4 + P2 / \mygraph{1[x=0,y=0, red] --[red] 2[x=1,y=0, red] -- 3[x=1,y=1] -- 4[x=0,y=1, red] --[red] 1} \ \mygraph{1[x=0,y=0, red] --[red] 2[x=0,y=1, red]}}},
        Sn + K2 / \mygraph{1[x=0,y=0, red] --[red] 2[x=0,y=1, red]} \ \mygraph{1[x=0,y=0.5, red] --[red] 2[x=1,y=1, red], 1 --[red] 3[x=1,y=0.5, red], 7[x=1,y=-1, rectangle, inner sep=0pt, fill=none, typeset=\smash{\footnotesize\vdots}]},
        K3 + K2 / \mygraph{1[x=0,y=0, red] --[red] 2[x=1,y=0, red] --[red] 3[x=0.5,y=1, red] --[red] 1} \ \mygraph{1[x=0,y=0, red] -- 2[x=0,y=1]}}},
        Sn <- {K3 one leaf, P4, K2 + K2, K3 + K2}
};

\node [rectangle, fill=black, opacity=0.08, fit=(K3 one leaf) (K3 two leaves) (K3 n leaves)] {};

\node [rectangle, fill=black, opacity=0.08, fit=(P4) (P3 two leaves) (P3 n leaves)] {};

\node [rectangle, fill=black, opacity=0.08, fit=(K2 + K2) (Sn + K2)] {};

\end{tikzpicture}
}
\caption{\footnotesize All graphs that dissipate. The arrows represent performing the jump graph operation. Isolated vertices are not drawn where they may appear. The $d$ value of a given graph is found by counting the number of arrows from that graph to the empty set. Whenever a graph contains a forbidden induced subgraph from Figure~\ref{fig:beineke_jump_graphs}, one such subgraph is highlighted in red.} \label{fig:dissipating_graphs}
\end{figure}
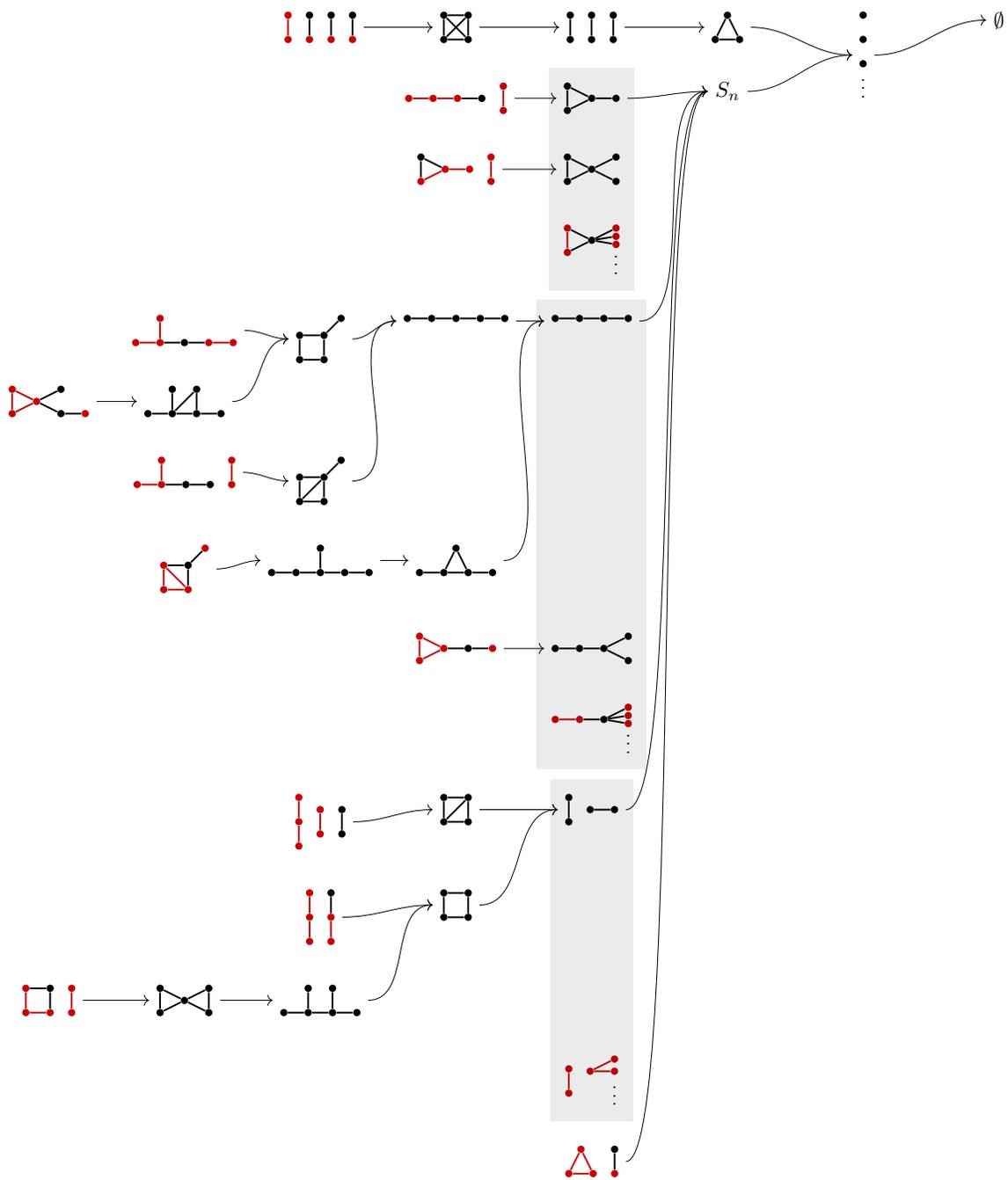

\section{Graphs with infinite $d$ value}\label{Section 4}

In this section, we turn to graphs with infinite $d$ value. In some sense, Figure \ref{fig:dissipating_graphs} already gives a characterization of $d$-infinite graphs-- if $G$ does not appear in Figure \ref{fig:dissipating_graphs}, then $d(G) = \infty$. However, there are more detailed questions we can ask about $d$-infinite graphs. In particular: what is the end behavior of the sequence $\{J^k(G)\}$? Is there a graph which satisfies $J^k(G) = G$ for some $k \geq 1$? In the following section, we will begin to answer these questions. This investigation extends beyond the results in \cite{Char} by asking about the quality of divergence for diverging jump graph sequences.

\subsection{Two Special Graphs}\label{Section 4.1}
As we have discussed earlier, the line graph $L(G)$ and the jump graph $J(G)$ are complements. This follows directly from the definition of each. Knowing this, we can apply the following result due to Aigner in ``Graphs whose complement and line graph are isomorphic'' \cite{Aig}. The graphs $C_5$ and the net graph are shown in Figure~\ref{fig:C5_and_N}.

\begin{thm}[\cite{Aig}]\label{Theorem: C5_P_first_jump}
    The only graphs satisfying
   $L(G)^c = G$ are $C_5$ and the net graph $N$.
\end{thm}

\noindent Applying this to iterated jump graphs, we have the following consequence.
\begin{cor}
    If $G$ is a graph, then $J^k(G) = G$ for all $k \geq 0$ if and only if $G$ is $C_5$ or $N$.
\end{cor}

\begin{rem}
This is a rephrasing of Theorem 4 in \cite{Char}.
\end{rem}

\begin{figure}[H]
  \centering
  \begin{minipage}[b]{0.3\textwidth}\centering
    \begin{tikzpicture}[line cap=round,line join=round,>=triangle 45,x=0.8cm,y=0.8cm]
    \draw [line width=1pt] (0,1)-- (-0.951,0.309);
    \draw [line width=1pt] (-0.951,0.309)-- (-0.588,-0.809);
    \draw [line width=1pt] (-0.588,-0.809)-- (0.588,-0.809);
    \draw [line width=1pt] (0.588,-0.809)-- (0.951,0.309);
    \draw [line width=1pt] (0.951,0.309)-- (0,1);
    \begin{scriptsize}
    \draw [fill=black] (0,1) circle (2pt);
    \draw [fill=black] (-0.951,0.309) circle (2pt);
    \draw [fill=black] (-0.588,-0.809) circle (2pt);
    \draw [fill=black] (0.588,-0.809) circle (2pt);
    \draw [fill=black] (0.951,0.309) circle (2pt);
    \end{scriptsize}
    \end{tikzpicture}
    \caption*{\large $C_5$}
  \end{minipage}
  \begin{minipage}[b]{0.3\textwidth}\centering
    \begin{tikzpicture}[line cap=round,line join=round,>=triangle 45,x=0.8cm,y=0.8cm]
    \draw [line width=1pt] (4,-1)-- (4.866025403784439,0.5);
    \draw [line width=1pt] (4.866025403784439,0.5)-- (3.134,0.5);
    \draw [line width=1pt] (3.134,0.5)-- (4,-1);
    \draw [line width=1pt] (2.5934,1.1)-- (3.134,0.5);
    \draw [line width=1pt] (5.40809,1.1)-- (4.866025403784439,0.5);
    \draw [line width=1pt] (4,-1)-- (4.00763,-1.8);
    \begin{scriptsize}
    \draw [fill=black] (4,-1) circle (2pt);
    \draw [fill=black] (4.866025403784439,0.5) circle (2pt);
    \draw [fill=black] (3.134,0.5) circle (2pt);
    \draw [fill=black] (2.5934, 1.1) circle (2pt);
    \draw [fill=black] (5.40809,1.1) circle (2pt);
    \draw [fill=black] (4.00763,-1.8) circle (2pt);
    \end{scriptsize}
    \end{tikzpicture}
    \caption*{\large $N$}
  \end{minipage}
  \caption{$C_5$ and the net graph $N$. These are the only graphs such that $G = J(G)$.}\label{fig:C5_and_N}
\end{figure}

In particular, we observe that $C_5$ and $N$ are both $d$-infinite. If $G$ has $C_5$ as a subgraph, then for every $k \geq 1$ we know $J^k(C_5) = C_5$ is a subgraph of $J^k(G)$ by Lemma~\ref{Lemma: J(H)_subgraph_of_J(G)} and the analogous statement is true for the net graph $N$. Given this, we can talk about \emph{accumulation}.

\begin{defn}
A graph $G$ \emph{accumulates} $C_5$ (resp.~$N$) if there is some $k$ such that $C_5 \subseteq J^k(G)$ (resp.~$N \subseteq J^k(G)$).
\end{defn}

Like with any $d$-infinite graph, if $G$ has $C_5$ or $N$ as a snipped subgraph, then it will have infinite $d$ value by Proposition~\ref{Prop: subgraph-d-value}. The surprising result is that the converse turns out to be true. If $d(G) = \infty$, then $G$ will accumulate $C_5$ or $N$. In the following two subsections, we will focus on proving this fact which will be completed in Theorem~\ref{Theorem: accumulating}.

\subsection{Some Useful Graphs}\label{Section 4.2}

The proofs in Section~\ref{Section 4.3} involve a decent amount of casework. In order to make this more manageable, we will introduce some useful graphs in addition to $C_5$ and $N$. Each of the following graphs has an infinite $d$ value and accumulates $C_5$. While performing casework in the next section, we can look for $C_5$, $N$, or one of these other useful graphs as a snipped subgraph. If a graph $G$ has one of these, then we know $d(G) = \infty$ and $G$ accumulates $C_5$ or $N$.

By Corollary~\ref{Corollary: snipped_d=infinity} and Lemma~\ref{Lemma: H_snipped_of_G, J(H)_in_J(G)}, if $H$ is a snipped subgraph of $G$ and $d(H) = \infty$, then $d(G) = \infty$ and $J(H) \subseteq J(G)$. We use this fact to show the following graphs have infinite $d$ value and accumulate $C_5$.

First, we have $\mathbf{K_{2,3}}$ in Figure~\ref{fig:K_2,3}. The jump graph of $K_{2,3}$ has $C_5$ as a snipped subgraph. We conclude $d(K_{2,3}) = \infty$ and $K_{2,3}$ accumulates $C_5$.

\begin{figure}[H]
  \centering
  \begin{minipage}[b]{0.3\textwidth}\centering
 \begin{tikzpicture}[line cap=round,line join=round,>=triangle 45,x=0.5cm,y=0.5cm]
    \draw [line width=1pt] (8,1)-- (10,0);
    \draw [line width=1pt] (8,0)-- (10,0);
    \draw [line width=1pt] (8,-1)-- (10,0);
    \draw [line width=1pt] (8,1)-- (6,0);
    \draw [line width=1pt] (6,0)-- (8,0);
    \draw [line width=1pt] (6,0)-- (8,-1);
    \begin{scriptsize}
    \draw [fill=black] (8,1) circle (1.8pt);
    \draw [fill=black] (10,0) circle (1.8pt);
    \draw [fill=black] (8,0) circle (1.8pt);
    \draw [fill=black] (8,-1) circle (1.8pt);
    \draw [fill=black] (6,0) circle (1.8pt);
\end{scriptsize}
\end{tikzpicture}
    \caption*{\large $K_{2,3}$}
  \end{minipage}
  \begin{minipage}[b]{0.3\textwidth}\centering
    \begin{tikzpicture}[line cap=round,line join=round,>=triangle 45,x=0.8cm,y=0.8cm]
    \draw [line width=1pt] (0,1)-- (-0.8660254037844385,0.5);
    \draw [line width=1pt] (-0.8660254037844385,0.5)-- (-0.8660254037844388,-0.5);
    \draw [line width=1pt] (-0.8660254037844388,-0.5)-- (0,-1);
    \draw [line width=1pt] (0,-1)-- (0.8660254037844384,-0.5);
    \draw [line width=1pt] (0.8660254037844384,-0.5)-- (0.8660254037844386,0.5);
    \draw [line width=1pt] (0.8660254037844386,0.5)-- (0,1);
    \begin{scriptsize}
    \draw [fill=black] (0,1) circle (1.8pt);
    \draw [fill=black] (-0.8660254037844385,0.5) circle (1.8pt);
    \draw [fill=black] (-0.8660254037844388,-0.5) circle (2pt);
    \draw [fill=black] (0,-1) circle (2pt);
    \draw [fill=black] (0.8660254037844384,-0.5) circle (2pt);
    \draw [fill=black] (0.8660254037844386,0.5) circle (2pt);
    \end{scriptsize}
    \end{tikzpicture}
    \caption*{\large $J(K_{2,3})$}
  \end{minipage}
  \caption{The graph $K_{2,3}$ and its jump graph.}\label{fig:K_2,3}
\end{figure}

\noindent Next, we have the \textbf{Bug Graph} in Figure~\ref{fig:Bug}. The jump graph of the Bug Graph has $K_{2, 3}$ as a snipped subgraph. We conclude the Bug Graph has $d(G) = \infty$ and accumulates $C_5$.

\begin{figure}[H]
  \centering
  \begin{minipage}[b]{0.3\textwidth}\centering
    \begin{tikzpicture}[line cap=round,line join=round,>=triangle 45,x=0.5cm,y=0.5cm]
    \draw [line width=1pt] (-1,2)-- (1,2);
    \draw [line width=1pt] (1,2)-- (1,0);
    \draw [line width=1pt] (1,0)-- (-1,0);
    \draw [line width=1pt] (-1,0)-- (-1,2);
    \draw [line width=1pt] (-1,2)-- (-1.46,3.42);
    \draw [line width=1pt] (-1,2)-- (-2.38,2.54);
    \begin{scriptsize}
    \draw [fill=black] (-1,2) circle (1.8pt);
    \draw [fill=black] (1,2) circle (1.8pt);
    \draw [fill=black] (1,0) circle (1.8pt);
    \draw [fill=black] (-1,0) circle (1.8pt);
    \draw [fill=black] (-1.46,3.42) circle (1.8pt);
    \draw [fill=black] (-2.38,2.54) circle (1.8pt);
    \draw [color=white,fill=white] (2.5,0) circle (1pt);
    \end{scriptsize}
    \end{tikzpicture}
    \caption*{\large $G$}
  \end{minipage}
  \begin{minipage}[b]{0.3\textwidth}\centering
    \begin{tikzpicture}[line cap=round,line join=round,>=triangle 45,x=0.5cm,y=0.5cm]
    \draw [line width=1pt] (5,2)-- (7,2);
    \draw [line width=1pt] (7,2)-- (7,0);
    \draw [line width=1pt] (7,0)-- (5,0);
    \draw [line width=1pt] (5,0)-- (5,2);
    \draw [line width=1pt] (7,0)-- (8.5,0);
    \draw [line width=1pt] (5,2)-- (3.62,2.54);
    \begin{scriptsize}
    \draw [fill=black] (5,2) circle (1.8pt);
    \draw [fill=black] (7,2) circle (1.8pt);
    \draw [fill=black] (7,0) circle (1.8pt);
    \draw [fill=black] (5,0) circle (1.8pt);
    \draw [fill=black] (8.5,0) circle (1.8pt);
    \draw [fill=black] (3.62,2.54) circle (1.8pt);
    \end{scriptsize}
    \end{tikzpicture}
    \caption*{\large $J(G)$}
  \end{minipage}
  \caption{The Bug Graph and its jump graph.}\label{fig:Bug}
\end{figure}

Next, we have the \textbf{Stickman Graph} in Figure~\ref{fig:Stickman}. The jump graph of the Stickman has $K_{2,3}$ as a subgraph. We conclude the Stickman has $d(G) = \infty$ and accumulates $C_5$.

\begin{figure}[H]
  \centering
  \begin{minipage}[b]{0.3\textwidth}\centering
    \begin{tikzpicture}[line cap=round,line join=round,>=triangle 45,x=0.55cm,y=0.55cm]
    \draw [line width=1pt] (3,5)-- (3,4);
    \draw [line width=1pt] (3,4)-- (2,4);
    \draw [line width=1pt] (3,4)-- (4,4);
    \draw [line width=1pt] (3,4)-- (3,3);
    \draw [line width=1pt] (3,3)-- (2.24,2.22);
    \draw [line width=1pt] (3,3)-- (3.76,2.22);
    \begin{scriptsize}
    \draw [fill=black] (3,5) circle (1.8pt);
    \draw [fill=black] (3,4) circle (1.8pt);
    \draw [fill=black] (2,4) circle (1.8pt);
    \draw [fill=black] (4,4) circle (1.8pt);
    \draw [fill=black] (3,3) circle (1.8pt);
    \draw [fill=black] (2.24,2.22) circle (1.8pt);
    \draw [fill=black] (3.76,2.22) circle (1.8pt);
    \end{scriptsize}
    \end{tikzpicture}
    \caption*{\large $G$}
  \end{minipage}
  \begin{minipage}[b]{0.3\textwidth}\centering
    \begin{tikzpicture}[line cap=round,line join=round,>=triangle 45,x=0.5cm,y=0.5cm]
    \draw [line width=1pt] (8,1)-- (10,0);
    \draw [line width=1pt] (8,0)-- (10,0);
    \draw [line width=1pt] (8,-1)-- (10,0);
    \draw [line width=1pt] (8,1)-- (6,0);
    \draw [line width=1pt] (6,0)-- (8,0);
    \draw [line width=1pt] (6,0)-- (8,-1);
    \begin{scriptsize}
    \draw [fill=black] (8,1) circle (1.8pt);
    \draw [fill=black] (10,0) circle (1.8pt);
    \draw [fill=black] (8,0) circle (1.8pt);
    \draw [fill=black] (8,-1) circle (1.8pt);
    \draw [fill=black] (6,0) circle (1.8pt);
    \draw [fill=black] (6,-1) circle (1.8pt);
    \end{scriptsize}
    \end{tikzpicture}
    \caption*{\large $J(G)$}
  \end{minipage}
  \caption{The Stickman Graph and its jump graph.}\label{fig:Stickman}
\end{figure}

Lastly, we have the \textbf{Pendulum Graph} in Figure~\ref{fig:Pendulum}. The jump graph of the Pendulum has $K_{2,3}$ as a snipped subgraph. We conclude the Pendulum has $d(G) = \infty$ and accumulates $C_5$.

\begin{figure}[H]
  \centering
  \begin{minipage}[b]{0.3\textwidth}\centering
    \begin{tikzpicture}[line cap=round,line join=round,>=triangle 45,x=0.55cm,y=0.55cm]
    \draw [line width=1pt] (3,4)-- (2,5);
    \draw [line width=1pt] (3,4)-- (4,5);
    \draw [line width=1pt] (3,4)-- (3,3);
    \draw [line width=1pt] (3,3)-- (2.24,2.22);
    \draw [line width=1pt] (3,3)-- (3.76,2.22);
    \draw [line width=1pt] (2.24,2.22) -- (3.76,2.22);
    \begin{scriptsize}
    \draw [fill=black] (2,5) circle (1.8pt);
    \draw [fill=black] (4,5) circle (1.8pt);
    \draw [fill=black] (3,4) circle (1.8pt);
    \draw [fill=black] (3,3) circle (1.8pt);
    \draw [fill=black] (2.24,2.22) circle (1.8pt);
    \draw [fill=black] (3.76,2.22) circle (1.8pt);
    \end{scriptsize}
    \end{tikzpicture}
    \caption*{\large $G$}
  \end{minipage}
  \begin{minipage}[b]{0.3\textwidth}\centering
    \begin{tikzpicture}[line cap=round,line join=round,>=triangle 45,x=0.5cm,y=0.5cm]
    \draw [line width=1pt] (8,1)-- (10,0);
    \draw [line width=1pt] (8,0)-- (10,0);
    \draw [line width=1pt] (8,-1)-- (10,0);
    \draw [line width=1pt] (8,1)-- (6,0);
    \draw [line width=1pt] (6,0)-- (8,0);
    \draw [line width=1pt] (6,0)-- (8,-1);
    \draw [line width=1pt] (6,-1)-- (8,-1);
    \begin{scriptsize}
    \draw [fill=black] (8,1) circle (1.8pt);
    \draw [fill=black] (10,0) circle (1.8pt);
    \draw [fill=black] (8,0) circle (1.8pt);
    \draw [fill=black] (8,-1) circle (1.8pt);
    \draw [fill=black] (6,0) circle (1.8pt);
    \draw [fill=black] (6,-1) circle (1.8pt);
    \end{scriptsize}
    \end{tikzpicture}
    \caption*{\large $J(G)$}
  \end{minipage}
  \caption{The Pendulum Graph and its jump graph.}\label{fig:Pendulum}
\end{figure}

\begin{rem}
All of the names for the above graphs (except for $K_{2,3}$) are original. Some of the graphs have been given different names in other contexts, but we decided to use our own because of their memorable---and entertaining---nature. Of particular note is that the Bug Graph is sometimes called $R$ because from a certain orientation it looks like the letter $R$.
\end{rem}

\subsection{Diameter and $d$-infinite graphs}\label{Section 4.3}

In order to show that every graph with infinite $d$ value accumulates $C_5$ or $N$, we will split all graphs into cases based on diameter. The \emph{diameter} of a graph $G$ is the length of the longest shortest path between any pair of vertices of $G$. The diameter of a disconnected graph is defined to be infinite.
\begin{lem}\label{Lemma: diam>=5}
If a connected graph $G$ has a diameter of 5 or more then $G$ is $d$-infinite and accumulates $C_5$.
\end{lem}

\begin{proof}
If $G$ has a diameter of at least 5 then there is some path in $G$ of length 5. Thus, we have $C_5$ as a snipped subgraph so $d(G) = \infty$ by Proposition~\ref{Prop: subgraph-d-value} and $C_5 \subseteq J(G)$ by Lemma~\ref{Lemma: H_snipped_of_G, J(H)_in_J(G)}. 
\end{proof}

\begin{lem}\label{Lemma: diam=4}
Suppose $G$ has a diameter of $4$. Then the following are equivalent:
\begin{enumerate}[label=(\roman*)]
    \item $G$ is $d$-infinite,
    \item $G$ accumulates $C_5$ or $N$, and 
    \item $G$ has $6$ or more edges.
\end{enumerate}
\end{lem}

\begin{proof}
Let $P$ be a longest shortest path in $G$ with a length of $4$ and let $P$ have endpoints $A$ and $B$.

If $G$ only has $4$ edges, then $G = P$. In this case $d(G)$ is finite (see Figure \ref{fig:dissipating_graphs}). If $G$ has $5$ edges, then it must be one of the graphs in Figure~\ref{fig:diam4-5edges}, up to isomorphism, because of the diameter constraint. Both these graphs have finite $d$ value, as they appear in Figure \ref{fig:dissipating_graphs}.

\begin{figure}
  \centering
  \begin{minipage}[b]{0.3\textwidth}\centering
    \begin{tikzpicture}[line cap=round,line join=round,>=triangle 45,x=0.6cm,y=0.6cm]
    \draw [line width=1pt] (-2,0)-- (0,0);
    \draw [line width=1pt] (0,0)-- (0.74,1.5);
    \draw [line width=1pt] (0.74,1.5)-- (2.58,1.3);
    \draw [line width=1pt] (2.58,1.3)-- (3.14,2.84);
    \draw [line width=1pt] (2.58,1.3)-- (1.82,2.56);
    \begin{scriptsize}
    \draw [fill=blue] (-2,0) circle (1.8pt);
    \draw[color=blue] (-2.4,0.29) node {A};
    \draw [fill=black] (0,0) circle (1.8pt);
    \draw [fill=black] (0.74,1.5) circle (1.8pt);
    \draw [fill=black] (2.58,1.3) circle (1.8pt);
    \draw [fill=red] (3.14,2.84) circle (1.8pt);
    \draw[color=red] (3.3,3.27) node {B};
    \draw [fill=black] (1.82,2.56) circle (1.8pt);
    \end{scriptsize}
    \end{tikzpicture}
  \end{minipage}
  \begin{minipage}[b]{0.3\textwidth}\centering
    \begin{tikzpicture}[line cap=round,line join=round,>=triangle 45,x=0.6cm,y=0.6cm]
    \draw [line width=1pt] (4,0)-- (6,0);
    \draw [line width=1pt] (6,0)-- (6.74,1.5);
    \draw [line width=1pt] (6.74,1.5)-- (8.58,1.3);
    \draw [line width=1pt] (8.58,1.3)-- (9.14,2.84);
    \draw [line width=1pt] (6.74,1.5)-- (7.56,0.14);
    \begin{scriptsize}
    \draw [fill=blue] (4,0) circle (1.8pt);
    \draw[color=blue] (3.6,0.29) node {A};
    \draw [fill=black] (6,0) circle (1.8pt);
    \draw [fill=black] (6.74,1.5) circle (1.8pt);
    \draw [fill=black] (8.58,1.3) circle (1.8pt);
    \draw [fill=red] (9.14,2.84) circle (1.8pt);
    \draw[color=red] (9.3,3.27) node {B};
    \draw [fill=black] (7.56,0.14) circle (1.8pt);
    \end{scriptsize}
    \end{tikzpicture}
  \end{minipage}
  \caption{Graphs with five edges and diameter 4.} \label{fig:diam4-5edges}
\end{figure}
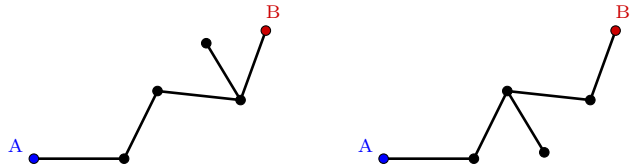

Now, let us consider when $G$ has $6$ or more edges. Then $G$ has a connected subgraph consisting of $P$ and two other edges. First, suppose that one of these edges has either $A$ or $B$ as an endpoint. This edge cannot have another endpoint in $P$. If it did, then we would contradict the assumption that $P$ is a shortest path from $A$ to $B$. Therefore, this additional edge and $P$ form a path of length 5 in $G$, so $C_5$ is a snipped subgraph of $G$. This implies $d(G) = \infty$ and $G$ accumulates $C_5$.

Thus, we turn our attention to the possibilities where neither additional edge is adjacent to $A$ or $B$. There are two cases: only one of two edges is adjacent to a vertex in $P$, or both edges are. 

In this first case, we immediately have $C_5$ as a snipped subgraph given by the path $P$ and the disjoint edge. One subgraph of this case is shown in Figure~\ref{fig:diam4_snipped_C5}, with the snipped $C_5$ in bold. So in this case, $G$ has an infinite $d$ value and accumulates $C_5$.

\begin{figure}
    \centering
    \begin{tikzpicture}[line cap=round,line join=round,>=triangle 45,x=0.8cm,y=0.8cm]
    \draw [line width=1.8pt] (4,0)-- (6,0);
    \draw [line width=1.8pt] (6,0)-- (6.74,1.5);
    \draw [line width=1.8pt] (6.74,1.5)-- (8.58,1.3);
    \draw [line width=1pt] (6.74,1.5)-- (7.56,0.14);
    \draw [line width=1.8pt] (9.14,2.84)-- (8.58,1.3);
    \draw [line width=1.8pt] (7.56,0.14)-- (8.3,-1.04);
    \begin{scriptsize}
    \draw [fill=blue] (4,0) circle (1.8pt);
    \draw[color=blue] (3.6,0.29) node {$A$};
    \draw [fill=black] (6,0) circle (1.8pt);
    \draw [fill=black] (6.74,1.5) circle (1.8pt);
    \draw [fill=black] (8.58,1.3) circle (1.8pt);
    \draw [fill=red] (9.14,2.84) circle (1.8pt);
    \draw[color=red] (9.3,3.27) node {B};
    \draw [fill=black] (7.56,0.14) circle (1.8pt);
    \draw [fill=black] (8.3,-1.04) circle (1.8pt);
    \end{scriptsize}
    \end{tikzpicture}
    \caption{If $G$ has two additional edges connected to $P$ and only one is adjacent to a vertex in $P$, then $G$ will have $C_5$ as a snipped subgraph.}
    \label{fig:diam4_snipped_C5}
\end{figure}
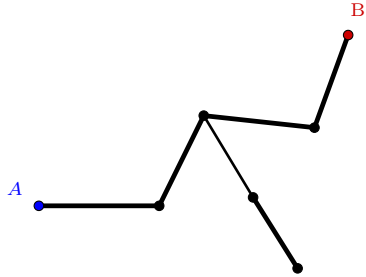

Now, we address the second case: suppose that each additional edge has an endpoint which lies in $P$. The possibilities are restricted by the assumption that $P$ is a shortest path between $A$ and $B$. Adding the edges cannot create an alternative path from $A$ to $B$ that has a length less than $4$. For instance, we cannot an edge to $P$ like the dashed edge in Figure~\ref{fig:diam4-shortenPathEx}.

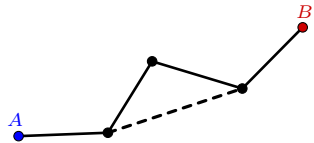
\begin{figure}
    \centering
    \begin{tikzpicture}[line cap=round,line join=round,>=triangle 45,x=0.7cm,y=0.7cm]
    \draw [line width=1pt] (-1.3950551209002424,1.6124844370643248)-- (0.31728942118441217,1.0955502356802764);
    \draw [line width=1pt] (0.31728942118441217,1.0955502356802764)-- (1.4642371805052656,2.258652188794385);
    \draw [line width=1pt] (-1.3950551209002424,1.6124844370643248)-- (-2.235073198149318,0.2555321584311982);
    \draw [line width=1pt] (-2.235073198149318,0.2555321584311982)-- (-3.9312635464407215,0.19091538325819216);
    \draw [line width=1.2pt, dashed] (-2.235073198149318,0.2555321584311982)-- (0.31728942118441217,1.0955502356802764);
    \begin{scriptsize}
    \draw [fill=black] (-1.3950551209002424,1.6124844370643248) circle (1.8pt);
    \draw [fill=black] (0.31728942118441217,1.0955502356802764) circle (1.8pt);
    \draw [fill=red] (1.4642371805052656,2.258652188794385) circle (1.8pt);
    \draw[color=red] (1.5,2.55) node {$B$};
    \draw [fill=blue] (-3.9312635464407215,0.19091538325819216) circle (1.8pt);
    \draw[color=blue] (-4,0.5) node {$A$};
    \draw [fill=black] (-2.235073198149318,0.2555321584311982) circle (1.8pt);
    \end{scriptsize}
    \end{tikzpicture}
    \caption{Adding the dashed edge above would shorten the distance from $A$ to $B$.} \label{fig:diam4-shortenPathEx}
\end{figure}

There are two ways to add the additional edges so that they form a cycle in $G$, and there are four ways to add two edges that do not form a cycle; all options are shown in Figure~\ref{fig:diam=4,6edges}. This can be confirmed by any exhaustive list of small graphs such as that in \cite{Steinbach}.

\begin{figure}
  \centering
  \begin{minipage}[b]{0.3\textwidth}\centering
    \begin{tikzpicture}[line cap=round,line join=round,>=triangle 45,x=0.5cm,y=0.5cm]
    \draw [line width=1pt] (10,0)-- (12,0);
    \draw [line width=1pt] (12,0)-- (12.74,1.5);
    \draw [line width=1pt] (12.74,1.5)-- (14.58,1.3);
    \draw [line width=1pt] (15.14,2.84)-- (14.58,1.3);
    \draw [line width=1pt] (12.74,1.5)-- (13.7,0.44);
    \draw [line width=1pt] (13.7,0.44)-- (12,0);
    \begin{scriptsize}
    \draw [fill=blue] (10,0) circle (1.8pt);
    \draw[color=blue] (9.6,0.29) node {$A$};
    \draw [fill=black] (12,0) circle (1.8pt);
    \draw [fill=black] (12.74,1.5) circle (1.8pt);
    \draw [fill=black] (14.58,1.3) circle (1.8pt);
    \draw [fill=red] (15.14,2.84) circle (1.8pt);
    \draw[color=red] (15.3,3.27) node {B};
    \draw [fill=black] (13.7,0.44) circle (1.8pt);
    \end{scriptsize}
    \end{tikzpicture}
    \caption*{\large (a)}
  \end{minipage}
  \begin{minipage}[b]{0.3\textwidth}\centering
    \begin{tikzpicture}[line cap=round,line join=round,>=triangle 45,x=0.5cm,y=0.5cm]
    \draw [line width=1pt] (16,0)-- (18,0);
    \draw [line width=1pt] (18,0)-- (18.74,1.5);
    \draw [line width=1pt] (18.74,1.5)-- (20.58,1.3);
    \draw [line width=1pt] (21.14,2.84)-- (20.58,1.3);
    \draw [line width=1pt] (18,0)-- (19.6,-0.04);
    \draw [line width=1pt] (19.6,-0.04)-- (20.58,1.3);
    \begin{scriptsize}
    \draw [fill=blue] (16,0) circle (1.8pt);
    \draw[color=blue] (15.6,0.29) node {$A$};
    \draw [fill=black] (18,0) circle (1.8pt);
    \draw [fill=black] (18.74,1.5) circle (1.8pt);
    \draw [fill=black] (20.58,1.3) circle (1.8pt);
    \draw [fill=red] (21.14,2.84) circle (1.8pt);
    \draw[color=red] (21.3,3.27) node {B};
    \draw [fill=black] (19.6,-0.04) circle (1.8pt);
    \end{scriptsize}
    \end{tikzpicture}
    \caption*{\large (b)}
  \end{minipage}
    \begin{minipage}[b]{0.3\textwidth}\centering
    \begin{tikzpicture}[line cap=round,line join=round,>=triangle 45,x=0.5cm,y=0.5cm]
    \draw [line width=1pt] (22,0)-- (24,0);
    \draw [line width=1pt] (24,0)-- (24.74,1.5);
    \draw [line width=1pt] (24.74,1.5)-- (26.58,1.3);
    \draw [line width=1pt] (24.74,1.5)-- (25.56,0.14);
    \draw [line width=1pt] (27.14,2.84)-- (26.58,1.3);
    \draw [line width=1pt] (24.74,1.5)-- (24.06,2.58);
    \begin{scriptsize}
    \draw [fill=blue] (22,0) circle (1.8pt);
    \draw[color=blue] (21.6,0.29) node {$A$};
    \draw [fill=black] (24,0) circle (1.8pt);
    \draw [fill=black] (24.74,1.5) circle (1.8pt);
    \draw [fill=black] (26.58,1.3) circle (1.8pt);
    \draw [fill=red] (27.14,2.84) circle (1.8pt);
    \draw[color=red] (27.3,3.27) node {B};
    \draw [fill=black] (25.56,0.14) circle (1.8pt);
    \draw [fill=black] (24.06,2.58) circle (1.8pt);
    \end{scriptsize}
    \end{tikzpicture}
    \caption*{\large (c)}
  \end{minipage}
  \begin{minipage}[b]{0.3\textwidth}\centering
    \begin{tikzpicture}[line cap=round,line join=round,>=triangle 45,x=0.5cm,y=0.5cm]
    \draw [line width=1pt] (-2,0)-- (0,0);
    \draw [line width=1pt] (0,0)-- (0.74,1.5);
    \draw [line width=1pt] (0.74,1.5)-- (2.58,1.3);
    \draw [line width=1pt] (2.58,1.3)-- (3.14,2.84);
    \draw [line width=1pt] (0,0)-- (-0.84,1.26);
    \draw [line width=1pt] (2.58,1.3)-- (1.82,2.56);
    \begin{scriptsize}
    \draw [fill=blue] (-2,0) circle (1.8pt);
    \draw[color=blue] (-2.4,0.29) node {$A$};
    \draw [fill=black] (0,0) circle (1.8pt);
    \draw [fill=black] (0.74,1.5) circle (1.8pt);
    \draw [fill=black] (2.58,1.3) circle (1.8pt);
    \draw [fill=red] (3.14,2.84) circle (1.8pt);
    \draw[color=red] (3.3,3.27) node {B};
    \draw [fill=black] (-0.84,1.26) circle (1.8pt);
    \draw[color=black] (-0.84, 1.66) node {$v_1$};
    \draw [fill=black] (1.82,2.56) circle (1.8pt);
    \draw[color=black] (1.82, 2.96) node {$v_2$};
    \end{scriptsize}
    \end{tikzpicture}
    \caption*{\large (d)}
  \end{minipage}
    \begin{minipage}[b]{0.3\textwidth}\centering
    \begin{tikzpicture}[line cap=round,line join=round,>=triangle 45,x=0.5cm,y=0.5cm]
    \draw [line width=1pt] (4,0)-- (6,0);
    \draw [line width=1pt] (6,0)-- (6.74,1.5);
    \draw [line width=1pt] (6.74,1.5)-- (8.58,1.3);
    \draw [line width=1pt] (8.58,1.3)-- (9.14,2.84);
    \draw [line width=1pt] (6,0)-- (5.16,1.26);
    \draw [line width=1pt] (6.74,1.5)-- (7.56,0.14);
    \begin{scriptsize}
    \draw [fill=blue] (4,0) circle (1.8pt);
    \draw[color=blue] (3.6,0.29) node {$A$};
    \draw [fill=black] (6,0) circle (1.8pt);
    \draw [fill=black] (6.74,1.5) circle (1.8pt);
    \draw [fill=black] (8.58,1.3) circle (1.8pt);
    \draw [color=black] (8.7, 0.9) node {$v_2$};
    \draw [fill=red] (9.14,2.84) circle (1.8pt);
    \draw[color=red] (9.3,3.27) node {B};
    \draw [fill=black] (5.16,1.26) circle (1.8pt);
    \draw [color=black] (5.16, 1.56) node {$v_1$};
    \draw [fill=black] (7.56,0.14) circle (1.8pt);
    \end{scriptsize}
    \end{tikzpicture}
    \caption*{\large (e)}
  \end{minipage}
    \begin{minipage}[b]{0.3\textwidth}\centering
    \begin{tikzpicture}[line cap=round,line join=round,>=triangle 45,x=0.5cm,y=0.5cm]
    \draw [line width=1pt] (10,0)-- (12,0);
    \draw [line width=1pt] (12,0)-- (12.74,1.5);
    \draw [line width=1pt] (12.74,1.5)-- (14.58,1.3);
    \draw [line width=1pt] (14.58,1.3)-- (15.14,2.84);
    \draw [line width=1pt] (12,0)-- (11.16,1.26);
    \draw [line width=1pt] (12,0)-- (12.76,-1.16);
    \begin{scriptsize}
    \draw [fill=blue] (10,0) circle (1.8pt);
    \draw[color=blue] (9.6,0.29) node {$A$};
    \draw [fill=black] (12,0) circle (1.8pt);
    \draw [fill=black] (12.74,1.5) circle (1.8pt);
    \draw [fill=black] (14.58,1.3) circle (1.8pt);
    \draw [fill=red] (15.14,2.84) circle (1.8pt);
    \draw[color=red] (15.3,3.27) node {B};
    \draw [fill=black] (11.16,1.26) circle (1.8pt);
    \draw [fill=black] (12.76,-1.16) circle (1.8pt);
    \end{scriptsize}
    \end{tikzpicture}
    \caption*{\large (f)}
  \end{minipage}
  \caption{There are six ways to add two edges to $P$ such that both are adjacent to a vertex in $P$ which is not $A$ or $B$ and so that $P$ remains a shortest path from $A$ to $B$.} \label{fig:diam=4,6edges}
  \end{figure}

We claim that each of these graphs has an infinite $d$ value and accumulates $C_5$ or $N$. First, graph (a) has $P_5$ as a subgraph and hence $C_5$ as a snipped subgraph. Next, both (b) and (d) have $K_{2,3}$ as a snipped subgraph. In (b) this can be seen by identifying vertices $A$ and $B$. In (d), identify $A$ and $B$ with each other and $v_1$ and $v_2$ with each other. Graphs (c) and (f) both have the Bug as a snipped subgraph which can be seen by identifying vertices $A$ and $B$. Lastly, graph (e) has the net graph $N$ as a snipped subgraph. To see this, identify the vertices labeled $v_1$ and $v_2$ with each other.

Thus, if $G$ has a diameter of 4 and at least 6 edges, it will have an infinite $d$ value and accumulate $C_5$ or $N$. All graphs with diameter 4 and fewer than 6 edges will dissipate and therefore never accumulate $C_5$ or $N$. We have then shown the equivalence of the statements.
\end{proof}

\begin{lem} \label{Lemma: diam=3}
Suppose $G$ has a diameter of 3. If $d(G) = \infty$, then it will accumulate $C_5$ or $N$. Furthermore, if $G$ has at least 7 edges, then $d(G) = \infty$ with the single exception of the family in Figure~\ref{fig:exception_family}.
\end{lem}

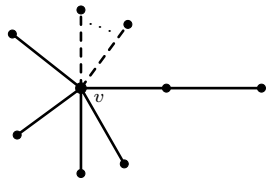
\begin{figure}[H]
    \centering
    \begin{tikzpicture}[line cap=round,line join=round,>=triangle 45,x=0.8cm,y=0.8cm]
    \draw [line width=1pt,dash pattern=on 3pt off 3pt] (0,0)-- (0,1.3);
    \draw [line width=1pt] (0,0)-- (1.42,0);
    \draw [line width=1pt] (0,0)-- (0,-1.42);
    \draw [line width=1pt] (1.42,0)-- (3,0);
    \draw [line width=1pt] (0,0)-- (-1.14,0.9);
    \draw [line width=1pt] (0,0)-- (-1.06,-0.78);
    \draw [line width=1pt,dash pattern=on 3pt off 3pt] (0,0)-- (0.78,1.06);
    \draw [line width=1pt] (0,0)-- (0.72,-1.26);
    \begin{scriptsize}
    \draw [fill=black] (0,0) circle (2pt);
    \draw[color=black] (0.3,-0.17) node {$v$};
    \draw [fill=black] (0,1.3) circle (1.5pt);
    \draw [fill=black] (1.42,0) circle (1.5pt);
    \draw [fill=black] (0,-1.42) circle (1.5pt);
    \draw [fill=black] (3,0) circle (1.5pt);
    \draw [fill=black] (-1.14,0.9) circle (1.5pt);
    \draw [fill=black] (-1.06,-0.78) circle (1.5pt);
    \draw [fill=black] (0.78,1.06) circle (1.5pt);
    \draw [fill=black] (0.72,-1.26) circle (1.5pt);
    \draw (0.35,1.0) node {\rotatebox{-18}{$\dots$}};
    \end{scriptsize}
    \end{tikzpicture}
    \caption{The only family graphs with diameter 3, at least 7 edges, and finite $d$-value. Any number of additional pendants can be added off vertex $v$.}
    \label{fig:exception_family}
\end{figure}

\begin{proof}
Throughout this proof, whenever we say that a graph has finite $d$ value ``by explicit calculation,'' one can look back to Figure~\ref{fig:dissipating_graphs} to find the relevant calculation. 

Let $P$ be a shortest path in $G$ with length $3$ and endpoints $A$ and $B$. If $G$ has $3$ edges, then $G = P$. This graph has a finite $d$ value by explicit calculation. If $G$ has $4$ edges, then it must be the graph in Figure~\ref{fig:diam=3,4edges}, up to isomorphism. This graph also has finite $d$ value by explicit calculation.

\begin{figure}
    \centering
    \begin{tikzpicture}[line cap=round,line join=round,>=triangle 45,x=0.6cm,y=0.6cm]
    \draw [line width=1pt] (-3,0)-- (-1,0);
    \draw [line width=1pt] (-1,0)-- (1,1);
    \draw [line width=1pt] (1,1)-- (3,1);
    \draw [line width=1pt] (1,1)-- (1.52,-0.48);
    \begin{scriptsize}
    \draw [fill=blue] (-3,0) circle (1.8pt);
    \draw [fill=black] (-1,0) circle (1.8pt);
    \draw [fill=black] (1,1) circle (1.8pt);
    \draw [fill=red] (3,1) circle (1.8pt);
    \draw [fill=black] (1.52,-0.48) circle (1.8pt);
    \end{scriptsize}
    \end{tikzpicture}
    \caption{The only graph with diameter 3 and 4 edges.}
    \label{fig:diam=3,4edges}
\end{figure}
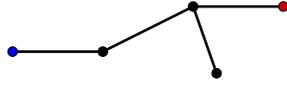

If $G$ has $5$ edges, then it can be one of five possible graphs shown in Figure~\ref{fig: diam=3, 5 edges}. One can see that these are all the graphs with diameter $3$ and $5$ edges by referencing a list of small graphs such as that given in \cite{Steinbach}. Each graph in Figure~\ref{fig: diam=3, 5 edges} has finite $d$ value as shown by explicit calculation.

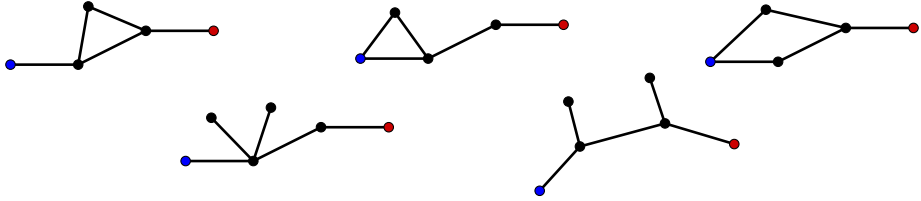
\begin{figure}
  \centering
  \begin{minipage}[b]{0.3\textwidth}\centering
    \begin{tikzpicture}[line cap=round,line join=round,>=triangle 45,x=0.45cm,y=0.45cm]
    \draw [line width=1pt] (-3,0)-- (-1,0);
    \draw [line width=1pt] (-1,0)-- (1,1);
    \draw [line width=1pt] (1,1)-- (3,1);
    \draw [line width=1pt,color=black] (-0.7,1.72)-- (-1,0);
    \draw [line width=1pt,color=black] (-0.7,1.72)-- (1,1);
    \begin{scriptsize}
    \draw [fill=blue] (-3,0) circle (1.8pt);
    \draw [fill=black] (-1,0) circle (1.8pt);
    \draw [fill=black] (1,1) circle (1.8pt);
    \draw [fill=red] (3,1) circle (1.8pt);
    \draw [fill=black] (-0.7,1.72) circle (1.8pt);
    \end{scriptsize}
    \end{tikzpicture}
  \end{minipage}
  \begin{minipage}[b]{0.3\textwidth}\centering
    \begin{tikzpicture}[line cap=round,line join=round,>=triangle 45,x=0.45cm,y=0.45cm]
    \draw [line width=1pt] (-3,0)-- (-1,0);
    \draw [line width=1pt] (-1,0)-- (1,1);
    \draw [line width=1pt] (1,1)-- (3,1);
    \draw [line width=1pt,color=black] (-3,0)-- (-1.98,1.36);
    \draw [line width=1pt,color=black] (-1.98,1.36)-- (-1,0);
    \begin{scriptsize}
    \draw [fill=blue] (-3,0) circle (1.8pt);
    \draw [fill=black] (-1,0) circle (1.8pt);
    \draw [fill=black] (1,1) circle (1.8pt);
    \draw [fill=red] (3,1) circle (1.8pt);
    \draw [fill=black] (-1.98,1.36) circle (1.8pt);
    \end{scriptsize}
    \end{tikzpicture}
  \end{minipage}
    \begin{minipage}[b]{0.3\textwidth}\centering
    \begin{tikzpicture}[line cap=round,line join=round,>=triangle 45,x=0.45cm,y=0.45cm]
    \draw [line width=1pt] (-3,0)-- (-1,0);
    \draw [line width=1pt] (-1,0)-- (1,1);
    \draw [line width=1pt] (1,1)-- (3,1);
    \draw [line width=1pt,color=black] (-1.36,1.54)-- (1,1);
    \draw [line width=1pt,color=black] (-1.36,1.54)-- (-3,0);
    \begin{scriptsize}
    \draw [fill=blue] (-3,0) circle (1.8pt);
    \draw [fill=black] (-1,0) circle (1.8pt);
    \draw [fill=black] (1,1) circle (1.8pt);
    \draw [fill=red] (3,1) circle (1.8pt);
    \draw [fill=black] (-1.36,1.54) circle (1.8pt);
    \end{scriptsize}
    \end{tikzpicture}
  \end{minipage} \begin{minipage}[b]{0.3\textwidth}\centering
    \begin{tikzpicture}[line cap=round,line join=round,>=triangle 45,x=0.45cm,y=0.45cm]
    \draw [line width=1pt] (-3,0)-- (-1,0);
    \draw [line width=1pt] (-1,0)-- (1,1);
    \draw [line width=1pt] (1,1)-- (3,1);
    \draw [line width=1pt] (-1,0)-- (-2.24,1.28);
    \draw [line width=1pt] (-0.48,1.58)-- (-1,0);
    \begin{scriptsize}
    \draw [fill=blue] (-3,0) circle (1.8pt);
    \draw [fill=black] (-1,0) circle (1.8pt);
    \draw [fill=black] (1,1) circle (1.8pt);
    \draw [fill=red] (3,1) circle (1.8pt);
    \draw [fill=black] (-2.24,1.28) circle (1.8pt);
    \draw [fill=black] (-0.48,1.58) circle (1.8pt);
    \end{scriptsize}
    \end{tikzpicture}
  \end{minipage}
  \begin{minipage}[b]{0.3\textwidth}\centering
    \begin{tikzpicture}[line cap=round,line join=round,>=triangle 45,x=0.5cm,y=0.5cm]
    \draw [line width=1pt] (-1.2496673767609794,0.15860699567168912)-- (1.0119197542942249,0.7724663598152464);
    \draw [line width=1pt] (1.0119197542942249,0.7724663598152464)-- (2.853497846724891,0.22322377084469516);
    \draw [line width=1pt] (-1.2496673767609794,0.15860699567168912)-- (-1.5565970588327571,1.3540173363723007);
    \draw [line width=1pt] (-1.2496673767609794,0.15860699567168912)-- (-2.32,-1.02);
    \draw [line width=1pt] (1.0119197542942249,0.7724663598152464)-- (0.6080649094629383,1.9840308943091094);
    \begin{scriptsize}
    \draw [fill=black] (-1.2496673767609794,0.15860699567168912) circle (1.8pt);
    \draw [fill=black] (1.0119197542942249,0.7724663598152464) circle (1.8pt);
    \draw [fill=red] (2.853497846724891,0.22322377084469516) circle (1.8pt);
    \draw [fill=black] (-1.5565970588327571,1.3540173363723007) circle (1.8pt);
    \draw [fill=blue] (-2.32,-1.02) circle (1.8pt);
    \draw [fill=black] (0.6080649094629383,1.9840308943091094) circle (1.8pt);
    \end{scriptsize}
    \end{tikzpicture}
  \end{minipage}
    \caption{All graphs with diameter 3 and 5 edges.}\label{fig: diam=3, 5 edges}
\end{figure}

Now, let us consider the graphs with diameter $3$ and exactly $6$ edges. The possibilities, up to isomorphism, are shown in Figure~\ref{fig:diam3-6edges} which can be confirmed by referencing \cite{Steinbach}.

\begin{figure}
    \centering
    \begin{minipage}{.3\textwidth}\centering
    \begin{tikzpicture}[line cap=round,line join=round,>=triangle 45,x=0.5cm,y=0.5cm]
    \draw [line width=1pt] (-3,0)-- (-1,0);
    \draw [line width=1pt] (-1,0)-- (1,1);
    \draw [line width=1pt] (1,1)-- (3,1);
    \draw [line width=1pt,color=black] (-1,0)-- (-2.24,1.28);
    \draw [line width=1pt,color=black] (-0.48,1.58)-- (-1,0);
    \draw [line width=1pt,color=black] (-1,0)-- (-1.66,-1.44);
    \begin{scriptsize}
    \draw [fill=white] (-3,0) circle (1.8pt);
    \draw[color=black] (-3.5,0) node {$v_1$};
    \draw [fill=black] (-1,0) circle (1.8pt);
    \draw[color=black] (-0.8,-0.35) node {$v_2$};
    \draw [fill=black] (1,1) circle (1.8pt);
    \draw[color=black] (1,0.5) node {$v_3$};
    \draw [fill=black] (3,1) circle (1.8pt);
    \draw[color=black] (3,0.5) node {$v_4$};
    \draw [fill=white] (-2.24,1.28) circle (1.8pt);
    \draw [fill=white] (-0.48,1.58) circle (1.8pt);
    \draw [fill=white] (-1.66,-1.44) circle (1.8pt);
    \end{scriptsize}
    \end{tikzpicture}
    \caption*{\large (a)}
    \end{minipage}
    \begin{minipage}{.3\textwidth}\centering
    \begin{tikzpicture}[line cap=round,line join=round,>=triangle 45,x=0.5cm,y=0.5cm]
    \draw [line width=1pt] (-3,0)-- (-1,0);
    \draw [line width=1pt] (-1,0)-- (1,1);
    \draw [line width=1pt] (1,1)-- (3,1);
    \draw [line width=1pt] (-1.92,1.52)-- (-1,0);
    \draw [line width=1pt] (-1.92,1.52)-- (-3,0);
    \draw [line width=1pt,color=black] (-1,0)-- (-0.2,1.88);
    \begin{scriptsize}
    \draw [fill=white] (-3,0) circle (1.8pt);
    \draw[color=black] (-3.5,0) node {$v_1$};
    \draw [fill=black] (-1,0) circle (1.8pt);
    \draw[color=black] (-1,-0.5) node {$v_2$};
    \draw [fill=black] (1,1) circle (1.8pt);
    \draw[color=black] (1,0.5) node {$v_4$};
    \draw [fill=black] (3,1) circle (1.8pt);
    \draw[color=black] (3,0.5) node {$v_5$};
    \draw [fill=white] (-1.92,1.52) circle (1.8pt);
    \draw [fill=black] (-0.2,1.88) circle (1.8pt);
    \draw[color=black](0.3,1.88) node {$v_3$};
    \end{scriptsize}
    \end{tikzpicture}
    \caption*{\large (b)}
    \end{minipage}
    \begin{minipage}{.3\textwidth}\centering
    \begin{tikzpicture}[line cap=round,line join=round,>=triangle 45,x=0.5cm,y=0.5cm]
    \draw [line width=1pt] (-3,0)-- (-1,0);
    \draw [line width=1pt] (-1,0)-- (1,1);
    \draw [line width=1pt] (1,1)-- (3,1);
    \draw [line width=1pt] (-0.88,1.82)-- (1,1);
    \draw [line width=1pt,color=black] (-1,0)-- (-2.44,1.1);
    \draw [line width=1pt] (-0.88,1.82)-- (-1,0);
    \begin{scriptsize}
    \draw [fill=white] (-3,0) circle (1.8pt);
    \draw[color=black] (-3.5,0) node {$v_1$};
    \draw [fill=black] (-1,0) circle (1.8pt);
    \draw[color=black] (-1,-0.5) node {$v_2$};
    \draw [fill=black] (1,1) circle (1.8pt);
    \draw[color=black] (1,0.5) node {$v_4$};
    \draw [fill=black] (3,1) circle (1.8pt);
    \draw[color=black] (3,0.5) node {$v_5$};
    \draw [fill=black] (-0.88,1.82) circle (1.8pt);
    \draw[color=black] (-0.88,2.3) node {$v_3$};
    \draw [fill=white] (-2.44,1.1) circle (1.8pt);
    \end{scriptsize}
    \end{tikzpicture}
    \caption*{\large (c)}
    \end{minipage}
    \begin{minipage}{.3\textwidth}\centering
    \begin{tikzpicture}[line cap=round,line join=round,>=triangle 45,x=0.5cm,y=0.5cm]
    \draw [line width=1pt] (-3,0)-- (-0.6,-0.4);
    \draw [line width=1pt] (-0.6,-0.4)-- (1,1);
    \draw [line width=1pt] (1,1)-- (3,1);
    \draw [line width=1pt] (-1.56,1.88)-- (1,1);
    \draw [line width=1pt] (-1.56,1.88)-- (-3,0);
    \draw [line width=1pt,color=black] (-1.56,1.88)-- (-0.6,-0.4);
    \begin{scriptsize}
    \draw [fill=black] (-3,0) circle (1.8pt);
    \draw[color=black] (-3.5,0) node {$v_1$};
    \draw [fill=white] (-0.6,-0.4) circle (1.8pt);
    \draw[color=black] (-0.6,-0.8) node {$v_2$};
    \draw [fill=black] (1,1) circle (1.8pt);
    \draw[color=black] (1.2,0.5) node {$v_3$};
    \draw [fill=black] (3,1) circle (1.8pt);
    \draw[color=black] (3,0.5) node {$v_4$};
    \draw [fill=white] (-1.56,1.88) circle (1.8pt);
    \end{scriptsize}
    \end{tikzpicture}
    \caption*{\large (d)}
    \end{minipage}
    \begin{minipage}{.3\textwidth}\centering
    \begin{tikzpicture}[line cap=round,line join=round,>=triangle 45,x=0.45cm,y=0.45cm]
    \draw [line width=1pt] (5,0)-- (7,0);
    \draw [line width=1pt] (7,0)-- (9,1);
    \draw [line width=1pt] (9,1)-- (11,1);
    \draw [line width=1pt,color=black] (7,0)-- (5.76,1.28);
    \draw [line width=1pt,color=black] (7.52,1.58)-- (7,0);
    \draw [line width=1pt,color=black] (9,1)-- (10.16,-0.28);
    \begin{scriptsize}
    \draw [fill=blue] (5,0) circle (1.8pt);
    \draw [fill=black] (7,0) circle (1.8pt);
    \draw [fill=black] (9,1) circle (1.8pt);
    \draw [fill=red] (11,1) circle (1.8pt);
    \draw [fill=black] (5.76,1.28) circle (1.8pt);
    \draw [fill=black] (7.52,1.58) circle (1.8pt);
    \draw [fill=black] (10.16,-0.28) circle (1.8pt);
    \end{scriptsize}
    \end{tikzpicture}
    \caption*{\large (e)}
    \end{minipage}
    \begin{minipage}{.3\textwidth}\centering
    \begin{tikzpicture}[line cap=round,line join=round,>=triangle 45,x=0.4cm,y=0.4cm]
    \draw [line width=1pt] (13,0)-- (15,0);
    \draw [line width=1pt] (15,0)-- (17,1);
    \draw [line width=1pt] (17,1)-- (19,1);
    \draw [line width=1pt] (14.44,1.88)-- (17,1);
    \draw [line width=1pt] (14.44,1.88)-- (13,0);
    \draw [line width=1pt,color=black] (17,1)-- (18.24,2.24);
    \begin{scriptsize}
    \draw [fill=blue] (13,0) circle (1.8pt);
    \draw [fill=black] (15,0) circle (1.8pt);
    \draw [fill=black] (17,1) circle (1.8pt);
    \draw [fill=red] (19,1) circle (1.8pt);
    \draw [fill=black] (14.44,1.88) circle (1.8pt);
    \draw [fill=black] (18.24,2.24) circle (1.8pt);
    \end{scriptsize}
    \end{tikzpicture}
    \caption*{\large (f)}
    \end{minipage}
    \begin{minipage}{.3\textwidth}\centering
    \begin{tikzpicture}[line cap=round,line join=round,>=triangle 45,x=0.4cm,y=0.4cm]
    \draw [line width=1pt] (-3,0)-- (-1,0);
    \draw [line width=1pt] (-1,0)-- (1,1);
    \draw [line width=1pt] (1,1)-- (3,1);
    \draw [line width=1pt] (-1.96,1.64)-- (-1,0);
    \draw [line width=1pt] (-1.96,1.64)-- (-3,0);
    \draw [line width=1pt,color=black] (1,1)-- (1.52,-0.48);
    \begin{scriptsize}
    \draw [fill=blue] (-3,0) circle (1.8pt);
    \draw [fill=black] (-1,0) circle (1.8pt);
    \draw [fill=black] (1,1) circle (1.8pt);
    \draw [fill=red] (3,1) circle (1.8pt);
    \draw [fill=black] (-1.96,1.64) circle (1.8pt);
    \draw [fill=black] (1.52,-0.48) circle (1.8pt);
    \end{scriptsize}
    \end{tikzpicture}
    \caption*{\large (g)}
    \end{minipage}
    \begin{minipage}{.3\textwidth}\centering
    \begin{tikzpicture}[line cap=round,line join=round,>=triangle 45,x=0.45cm,y=0.45cm]
    \draw [line width=1pt] (-3,0)-- (-0.6,-0.4);
    \draw [line width=1pt] (-0.6,-0.4)-- (1,1);
    \draw [line width=1pt] (1,1)-- (3,1);
    \draw [line width=1pt] (-1.56,1.88)-- (1,1);
    \draw [line width=1pt] (-1.56,1.88)-- (-3,0);
    \draw [line width=1pt,color=black] (-0.6,-0.4)-- (1.68,-0.34);
    \begin{scriptsize}
    \draw [fill=blue] (-3,0) circle (1.8pt);
    \draw [fill=black] (-0.6,-0.4) circle (1.8pt);
    \draw [fill=black] (1,1) circle (1.8pt);
    \draw [fill=red] (3,1) circle (1.8pt);
    \draw [fill=black] (-1.56,1.88) circle (1.8pt);
    \draw [fill=black] (1.68,-0.34) circle (1.8pt);
    \end{scriptsize}
    \end{tikzpicture}
    \caption*{\large (h)}
    \end{minipage}
    \begin{minipage}{.3\textwidth}\centering
    \begin{tikzpicture}[line cap=round,line join=round,>=triangle 45,x=0.4cm,y=0.4cm]
    \draw [line width=1pt] (5,0)-- (7,0);
    \draw [line width=1pt] (7,0)-- (9,1);
    \draw [line width=1pt] (9,1)-- (11,1);
    \draw [line width=1pt] (7.16,1.8)-- (7,0);
    \draw [line width=1pt] (7.16,1.8)-- (9,1);
    \draw [line width=1pt,color=black] (7.16,1.8)-- (6.44,3.36);
    \begin{scriptsize}
    \draw [fill=blue] (5,0) circle (1.8pt);
    \draw [fill=black] (7,0) circle (1.8pt);
    \draw [fill=black] (9,1) circle (1.8pt);
    \draw [fill=red] (11,1) circle (1.8pt);
    \draw [fill=black] (7.16,1.8) circle (1.8pt);
    \draw [fill=black] (6.44,3.36) circle (1.8pt);
    \end{scriptsize}
    \end{tikzpicture}
    \caption*{\large (i)}
    \end{minipage}
    \begin{minipage}{.3\textwidth}\centering
    \begin{tikzpicture}[line cap=round,line join=round,>=triangle 45,x=0.7cm,y=0.7cm]
    \draw [line width=1pt] (0,1)-- (-0.951,0.309);
    \draw [line width=1pt] (-0.951,0.309)-- (-0.588,-0.809);
    \draw [line width=1pt] (-0.588,-0.809)-- (0.588,-0.809);
    \draw [line width=1pt] (0.588,-0.809)-- (0.951,0.309);
    \draw [line width=1pt] (0.951,0.309)-- (0,1);
    \draw [line width=1pt] (0.951,0.309)-- (1.8389767872148006,0.6398988690399297);
    \begin{scriptsize}
    \draw [fill=black] (0,1) circle (1.8pt);
    \draw [fill=black] (-0.951,0.309) circle (1.8pt);
    \draw [fill=blue] (-0.588,-0.809) circle (1.8pt);
    \draw [fill=black] (0.588,-0.809) circle (1.8pt);
    \draw [fill=black] (0.951,0.309) circle (1.8pt);
    \draw [fill=red] (1.8389767872148006,0.6398988690399297) circle (1.8pt);
    \end{scriptsize}
\end{tikzpicture}
    \caption*{\large (j)}
    \end{minipage}
    \begin{minipage}{.3\textwidth}\centering
    \begin{tikzpicture}[line cap=round,line join=round,>=triangle 45,x=0.8cm,y=0.8cm]
    \draw [line width=1pt] (0,1)-- (-0.8660254037844385,0.5);
    \draw [line width=1pt] (-0.8660254037844385,0.5)-- (-0.8660254037844388,-0.5);
    \draw [line width=1pt] (-0.8660254037844388,-0.5)-- (0,-1);
    \draw [line width=1pt] (0,-1)-- (0.8660254037844384,-0.5);
    \draw [line width=1pt] (0.8660254037844384,-0.5)-- (0.8660254037844386,0.5);
    \draw [line width=1pt] (0.8660254037844386,0.5)-- (0,1);
    \begin{scriptsize}
    \draw [fill=black] (0,1) circle (1.8pt);
    \draw [fill=black] (-0.8660254037844385,0.5) circle (1.8pt);
    \draw [fill=blue] (-0.8660254037844388,-0.5) circle (2pt);
    \draw [fill=black] (0,-1) circle (2pt);
    \draw [fill=black] (0.8660254037844384,-0.5) circle (2pt);
    \draw [fill=red] (0.8660254037844386,0.5) circle (2pt);
    \end{scriptsize}
    \end{tikzpicture}
    \caption*{\large (k)}
    \end{minipage}
    \caption{All graphs with a diameter of 3 and 6 edges. Graphs (a) - (d) have finite $d$ value while graphs (e) - (k) have infinite $d$ value.}
    \label{fig:diam3-6edges}
\end{figure}

Graphs (a) through (d) in Figure~\ref{fig:diam3-6edges} have a finite $d$ value which can be seen by explicit calculation. Graphs (e) through (k) each have infinite $d$ value and accumulate $C_5$ or $N$: (e) is the Stickman, (f) is the Bug Graph, (g) is the Pendulum, (h) has $P_5$ as a subgraph, (i) is the net graph $N$, (j) has $C_5$ as a subgraph, and (k) has $P_5$ as a subgraph. Therefore, if $G$ has 6 edges then it either has finite $d$ value, or $d(G) = \infty$ and $G$ accumulates $C_5$ or $N$.

Now, suppose that $G$ has diameter $3$ and $7$ or more edges. Then $G$ must have a connected subgraph containing 6 edges: $P$ and $3$ additional edges. If this subgraph has a diameter of $5$ or greater, then it will contain $C_5$ as a snipped subgraph. This means $d(G) = \infty$ and $G$ accumulates $C_5$. If this subgraph has diameter $4$, then by Lemma \ref{Lemma: diam=4} we conclude that this subgraph---and consequently $G$---will have an infinite $d$ value and accumulate $C_5$ or $N$. If this subgraph has diameter $3$, then it must be one of the eleven graphs shown in Figure \ref{fig:diam3-6edges}. The subgraph cannot have diameter less than $3$ because we are assuming that $P$ is a shortest path between $A$ and $B$.

If the subgraph of $G$ is isomorphic to graphs (e) through (k) in Figure~\ref{fig:diam3-6edges}, we have already shown that it is $d$-infinite and accumulates $C_5$ or $N$, so by Lemmas~\ref{Corollary: snipped_d=infinity} and \ref{Lemma: J(H)_subgraph_of_J(G)}, $d(G) = \infty$ and $G$ accumulates $C_5$ or $N$. Now, we will show that if $G$ \emph{strictly} contains any of the graphs (a) through (d), it will have $d = \infty$ and accumulate $C_5$ or $N$, except if $G$ is part of the family shown in Figure~\ref{fig:exception_family}. 

Before beginning the casework, we make two observations.
\begin{enumerate}
    \item Consider the graph (a) as shown in Figure~\ref{fig:diam3-6edges} and notice that each of the white vertices are identical under graph automorphism. Thus, without loss of generality, we will consider additions to graph (a) involving $v_1$ as a representative for all cases involving one white vertex since all other white vertices will follow the same logic as $v_1$.
    \item Now consider two graphs formed by adding an edge to graph (a): in the first, connect any non-adjacent vertices $v_i$ and $v_j$ with an edge; in the second, add a pendant edge to $v_i$. The second graph has the first graph as a snipped subgraph. Because of this fact, we need not consider adding pendant edges to the graph unless there is a vertex which cannot be connected to any other vertex in the graph.
\end{enumerate}
As we now examine graphs (a) through (d) from Figure~\ref{fig:diam3-6edges} separately, we will apply these two observations in the casework for each graph.

First, we look at graph (a). Adding an edge between $v_1$ and $v_4$ will give the Bug as a subgraph, and adding an edge between $v_1$ and and $v_3$ gives the Stickman as a subgraph. If we add an edge between two white vertices, we again acquire the Bug as a snipped subgraph. Adding an edge between $v_2$ and $v_4$ gives a graph with finite $d$ value, but with diameter $2$. This means that $G$ must strictly contain this graph and fall into another case. Finally, we consider adding a pendant to $v_2$. This graph has finite $d$ value and is part of the one exceptional family in Figure~\ref{fig:exception_family}. So if $G$ contains (a) as a strict subgraph and is not part of the family given in Figure~\ref{fig:exception_family}, then $d(G) = \infty$ and $G$ accumulates $C_5$. 

Next, we consider graph (b). Adding an edge between $v_1$ and $v_3$ or $v_1$ and $v_4$ gives a path of length 5. An additional edge between $v_1$ and $v_5$ gives $C_5$ as a subgraph. Adding an edge between $v_2$ and $v_5$ results in the Bowtie Graph with a pendant off the center vertex. The jump graph of this graph has the Stickman as a subgraph (see graph (viii) and its jump graph in Figure~\ref{fig:dissipate_diam2}). If we add an edge from $v_3$ to $v_4$, we have a path of length 5. Finally, an edge between $v_3$ and $v_5$ results in a graph with the Bug as a subgraph. Thus, if $G$ contains graph (b) as a strict subgraph, it will have infinite $d$ value and accumulate $C_5$.

Now we consider ways that $G$ can have graph (c) as a strict subgraph. Adding an edge between two white vertices will give a path of length 5, and adding an edge between $v_1$ and $v_3$ also gives a path of length 5. If we connect $v_1$ and $v_4$ with an edge, we have $K_{2,3}$ as a snipped subgraph. Connecting $v_1$ and $v_5$ with an edge gives $C_5$ as a subgraph. Lastly, if we add an edge between $v_2$ and $v_5$ or $v_3$ and $v_5$ we have the Bug as a subgraph. This tells us that if $G$ contains graph (c) as a strict subgraph, it will be $d$-infinite and accumulate $C_5$.

Lastly, we check the graph (d) case. If we add an edge between $v_1$ and $v_3$, the resulting graph is $K_4$ with a pendant edge off one vertex. The jump graph of this graph has $K_{2,3}$ as a snipped subgraph. Connecting $v_1$ and $v_4$ with an edge gives $K_{2,3}$ as a subgraph. Finally, if we add an edge between $v_2$ and $v_4$, we have $C_5$ as a subgraph. Hence, any $G$ with graph (d) as a strict subgraph will have $d(G) = \infty$ and accumulate $C_5$.

This concludes the proof for the diameter $3$ case. If $G$ strictly contains any of graphs (a) through (d) and $G$ is not part of the family given in Figure~\ref{fig:exception_family}, then it will have infinite $d$ value and accumulate $C_5$. 
\end{proof}

\begin{lem} \label{Lemma: diam=2}
Suppose $G$ has a diameter of 2. If $d(G) = \infty$, then $G$ will accumulate $C_5$. Furthermore, if $d(G)$ is finite then $G$ is one of the graphs in Figure~\ref{fig:dissipate_diam2}.
\end{lem}

\begin{figure}[H]
  \centering
  \begin{minipage}[b]{0.2\textwidth}\centering
    \begin{tikzpicture}[line cap=round,line join=round,>=triangle 45,x=0.5cm,y=0.5cm]
    \draw [line width=1pt] (-1,2)-- (1,2);
    \draw [line width=1pt] (1,2)-- (1,0);
    \draw [line width=1pt] (-1,0)-- (1,0);
    \draw [line width=1pt] (-1,2)-- (-1,0);
    \begin{scriptsize}
    \draw [fill=black] (-1,2) circle (1.8pt);
    \draw [fill=black] (1,2) circle (1.8pt);
    \draw [fill=black] (1,0) circle (1.8pt);
    \draw [fill=black] (-1,0) circle (1.8pt);
    \end{scriptsize}
    \end{tikzpicture}
    \caption*{\large $C_4$}
  \end{minipage}
    \begin{minipage}[b]{0.2\textwidth}\centering
    \begin{tikzpicture}[line cap=round,line join=round,>=triangle 45,x=0.45cm,y=0.45cm]
    \draw [line width=1pt] (-2,1)-- (-2,-1);
    \draw [line width=1pt] (0,0)-- (-2,-1);
    \draw [line width=1pt] (-2,1)-- (0,0);
    \draw [line width=1pt] (0,0)-- (2,1);
    \draw [line width=1pt] (2,1)-- (2,-1);
    \draw [line width=1pt] (2,-1)-- (0,0);
    \begin{scriptsize}
    \draw [fill=black] (-2,1) circle (1.8pt);
    \draw [fill=black] (-2,-1) circle (1.8pt);
    \draw [fill=black] (0,0) circle (1.8pt);
    \draw [fill=black] (2,1) circle (1.8pt);
    \draw [fill=black] (2,-1) circle (1.8pt);
    \end{scriptsize}
    \end{tikzpicture}
    \caption*{\large Bowtie Graph}
  \end{minipage}
    \begin{minipage}[b]{0.2\textwidth}\centering
    \begin{tikzpicture}[line cap=round,line join=round,>=triangle 45,x=0.4cm,y=0.4cm]
    \draw [line width=1pt,dash pattern=on 2pt off 2pt] (0,0)-- (-2,-1);
    \draw [line width=1pt] (0,2)-- (0,0);
    \draw [line width=1pt,dash pattern=on 2pt off 2pt] (0,0)-- (2,1);
    \draw [line width=1pt,dash pattern=on 2pt off 2pt] (2,-1)-- (0,0);
    \draw [line width=1pt,dash pattern=on 2pt off 2pt] (0,0)-- (0,-2);
    \draw [line width=1pt] (0,0)-- (-2,1);
    \begin{scriptsize}
    \draw [fill=black] (0,2) circle (1.8pt);
    \draw [fill=black] (-2,-1) circle (1.8pt);
    \draw [fill=black] (0,0) circle (1.8pt);
    \draw [fill=black] (2,1) circle (1.8pt);
    \draw [fill=black] (2,-1) circle (1.8pt);
    \draw [fill=black] (0,-2) circle (1.8pt);
    \draw [fill=black] (-2,1) circle (1.8pt);
    \end{scriptsize}
    \end{tikzpicture}
    \caption*{\large $S_n$}
  \end{minipage}
  \\ \vspace{20pt}
  \begin{minipage}[b]{0.2\textwidth}\centering
    \begin{tikzpicture}[line cap=round,line join=round,>=triangle 45,x=0.5cm,y=0.5cm]
    \draw [line width=1pt] (-1,2)-- (1,2);
    \draw [line width=1pt] (1,2)-- (1,0);
    \draw [line width=1pt] (-1,0)-- (1,0);
    \draw [line width=1pt] (-1,2)-- (-1,0);
    \draw [line width=1pt] (-1,2)-- (1,0);
    \begin{scriptsize}
    \draw [fill=black] (-1,2) circle (1.8pt);
    \draw [fill=black] (1,2) circle (1.8pt);
    \draw [fill=black] (1,0) circle (1.8pt);
    \draw [fill=black] (-1,0) circle (1.8pt);
    \end{scriptsize}
    \end{tikzpicture}
    \caption*{\large I}
  \end{minipage}
    \begin{minipage}[b]{0.2\textwidth}\centering
    \begin{tikzpicture}[line cap=round,line join=round,>=triangle 45,x=0.5cm,y=0.5cm]
    \draw [line width=1pt] (4,2)-- (6,2);
    \draw [line width=1pt] (6,2)-- (6,0);
    \draw [line width=1pt] (4,0)-- (6,0);
    \draw [line width=1pt] (4,2)-- (4,0);
    \draw [line width=1pt] (2.5,0)-- (4,0);
    \draw [line width=1pt] (6,2)-- (4,0);
    \begin{scriptsize}
    \draw [fill=black] (4,2) circle (1.8pt);
    \draw [fill=black] (6,2) circle (1.8pt);
    \draw [fill=black] (4,0) circle (1.8pt);
    \draw [fill=black] (2.5,0) circle (1.8pt);
    \draw [fill=black] (6,0) circle (1.8pt);
    \end{scriptsize}
    \end{tikzpicture}
    \caption*{\large II}
  \end{minipage}
    \begin{minipage}[b]{0.2\textwidth}\centering
    \begin{tikzpicture}[line cap=round,line join=round,>=triangle 45,x=0.5cm,y=0.5cm]
    \draw [line width=1pt] (-6,0)-- (-5,1);
    \draw [line width=1pt] (-5,1)-- (-4,0);
    \draw [line width=1pt] (-4,0)-- (-6,0);
    \draw [line width=1pt,dash pattern=on 2pt off 2pt] (-6,0)-- (-6,-1);
    \draw [line width=1pt] (-7.5,0)-- (-6,0);
    \draw [line width=1pt,dash pattern=on 2pt off 2pt] (-7,-1)-- (-6,0);
    \begin{scriptsize}
    \draw [fill=black] (-6,0) circle (1.8pt);
    \draw [fill=black] (-5,1) circle (1.8pt);
    \draw [fill=black] (-4,0) circle (1.8pt);
    \draw [fill=black] (-7,-1) circle (1.8pt);
    \draw [fill=black] (-7.5,0) circle (1.8pt);
    \draw [fill=black] (-6,-1) circle (1.8pt);
    \end{scriptsize}
    \end{tikzpicture}
    \caption*{\large III}
  \end{minipage}
  \caption{All graphs and families of graphs with a diameter of 2 which dissipate.}\label{fig:dissipate_diam2}
\end{figure}

\begin{proof}
We begin by pointing the reader to Figure~\ref{fig:dissipating_graphs} to check that all graphs in Figure~\ref{fig:dissipate_diam2} have finite $d$ value.

If $G$ has no cycles then it must be a star graph. All star graphs have a finite $d$ value, namely $d(S_n) = 2$. The family of star graphs is listed in Figure~\ref{fig:dissipate_diam2}.

If $G$ has a cycle $C_n$ with $n \geq 5$, then $G$ will have $C_5$ as a snipped subgraph. By applying Lemma \ref{Lemma: H_snipped_of_G, J(H)_in_J(G)}, we know that $d(G) = \infty$ and $G$ will accumulate $C_5$.

Now suppose that $G$ has $C_4$ as a subgraph. If $G = C_4$, then it has finite $d$ value and is listed in Figure~\ref{fig:dissipate_diam2}. If $C_4 \subseteq G$ and $G$ has exactly 5 edges then $G$ must be graph I in Figure~\ref{fig:dissipate_diam2} due to the diameter constraint, so again $G$ has finite $d$ value. 

We then suppose that $C_4 \subseteq G$ and $G$ has exactly $6$ edges. All options for $G$ are found by considering ways to add two edges to $C_4$ while maintaining a diameter of $2$. We find that $G$ is one of graphs (i), (ii), or (iii) in Figure~\ref{fig:diam2_C4}. Graph (i) has finite $d$ value and appears in Figure~\ref{fig:dissipate_diam2} as graph II. Graph (ii) has $C_5$ as a subgraph and graph (iii) is $K_{2,3}$ so both of these graphs have infinite $d$ value and accumulate $C_5$. 

Next suppose that $C_4 \subseteq G$ and $G$ has exactly 7 edges. If $G$ has graph (ii) or graph (iii) as a subgraph, then it will have infinite $d$ value and accumulate $C_5$. We consider all the ways to add three edges to $C_4$ without obtaining one of these subgraphs. There are two possibilities for $G$: graph (iv) and graph (v) in Figure~\ref{fig:diam2_C4}. The jump graph of (iv) has $K_{2,3}$ as a snipped subgraph so (iv) has infinite $d$ value and accumulates $C_5$. Graph (v) has the Bug as a subgraph so (v) also has infinite $d$ value and accumulates $C_5$. 

Finally, suppose that $C_4 \subseteq G$ and $G$ has 8 or more edges. We consider all the ways to add four edges to $C_4$ while maintaining a diameter of $2$. Any way we do this, we acquire one of graphs (ii), (iii), (iv), or (v) as a subgraph. Thus, $G$ will have infinite $d$ value and accumulate $C_5$.

\begin{figure}
  \centering
  \begin{minipage}[b]{0.2\textwidth}\centering
  \begin{tikzpicture}[line cap=round,line join=round,>=triangle 45,x=0.7cm,y=0.7cm]
    \draw [line width=1pt] (4,2)-- (6,2);
    \draw [line width=1pt] (6,2)-- (6,0);
    \draw [line width=1pt] (4,0)-- (6,0);
    \draw [line width=1pt] (4,2)-- (4,0);
    \draw [line width=1pt] (2.5,0)-- (4,0);
    \draw [line width=1pt] (6,2)-- (4,0);
    \begin{scriptsize}
    \draw [fill=black] (4,2) circle (1.8pt);
    \draw [fill=black] (6,2) circle (1.8pt);
    \draw [fill=black] (4,0) circle (1.8pt);
    \draw [fill=black] (2.5,0) circle (1.8pt);
    \draw [fill=black] (6,0) circle (1.8pt);
    \end{scriptsize}
    \end{tikzpicture}
    \caption*{\large (i)}
  \end{minipage}
 \begin{minipage}[b]{0.2\textwidth}\centering
     \begin{tikzpicture}[line cap=round,line join=round,>=triangle 45,x=0.7cm,y=0.7cm]
    \draw [line width=1pt] (8,2)-- (10,2);
    \draw [line width=1pt] (10,2)-- (10,0);
    \draw [line width=1pt] (8,0)-- (10,0);
    \draw [line width=1pt] (8,2)-- (8,0);
    \draw [line width=1pt] (10,2)-- (11,1);
    \draw [line width=1pt] (11,1)-- (10,0);
    \begin{scriptsize}
    \draw [fill=black] (8,2) circle (1.8pt);
    \draw [fill=black] (10,2) circle (1.8pt);
    \draw [fill=black] (10,0) circle (1.8pt);
    \draw [fill=black] (8,0) circle (1.8pt);
    \draw [fill=black] (11,1) circle (1.8pt);
    \end{scriptsize}
    \end{tikzpicture}
    \caption*{\large (ii)}
  \end{minipage}           \begin{minipage}[b]{0.2\textwidth}\centering
    \begin{tikzpicture}[line cap=round,line join=round,>=triangle 45,x=0.7cm,y=0.7cm]
    \draw [line width=1pt] (12,2)-- (14,2);
    \draw [line width=1pt] (14,2)-- (14,0);
    \draw [line width=1pt] (12,0)-- (14,0);
    \draw [line width=1pt] (12,2)-- (12,0);
    \draw [line width=1pt] (12,0)-- (13,1);
    \draw [line width=1pt] (13,1)-- (14,2);
    \begin{scriptsize}
    \draw [fill=black] (12,2) circle (1.8pt);
    \draw [fill=black] (14,2) circle (1.8pt);
    \draw [fill=black] (14,0) circle (1.8pt);
    \draw [fill=black] (12,0) circle (1.8pt);
    \draw [fill=black] (13,1) circle (1.8pt);
    \end{scriptsize}
    \end{tikzpicture}
    \caption*{\large (iii)}
  \end{minipage} \\
  \vspace{10pt}
  \begin{minipage}[b]{0.2\textwidth}\centering
    \begin{tikzpicture}[line cap=round,line join=round,>=triangle 45,x=0.7cm,y=0.7cm]
    \draw [line width=1pt] (0,2)-- (2,2);
    \draw [line width=1pt] (2,2)-- (2,0);
    \draw [line width=1pt] (0,0)-- (2,0);
    \draw [line width=1pt] (0,2)-- (0,0);
    \draw [line width=1pt] (-1.5,0)-- (0,0);
    \draw [line width=1pt] (2,2)-- (0,0);
    \draw [line width=1pt] (2,0)-- (0,2);
    \begin{scriptsize}
    \draw [fill=black] (0,2) circle (1.8pt);
    \draw [fill=black] (2,2) circle (1.8pt);
    \draw [fill=black] (0,0) circle (1.8pt);
    \draw [fill=black] (2,0) circle (1.8pt);
    \draw [fill=black] (-1.5,0) circle (1.8pt);
    \end{scriptsize}
    \end{tikzpicture}
    \caption*{\large (iv)}
  \end{minipage}
    \begin{minipage}[b]{0.2\textwidth}\centering
    \begin{tikzpicture}[line cap=round,line join=round,>=triangle 45,x=0.7cm,y=0.7cm]
    \draw [line width=1pt] (0,2)-- (2,2);
    \draw [line width=1pt] (2,2)-- (2,0);
    \draw [line width=1pt] (0,0)-- (2,0);
    \draw [line width=1pt] (0,2)-- (0,0);
    \draw [line width=1pt] (-1.5,0)-- (0,0);
    \draw [line width=1pt] (2,2)-- (0,0);
    \draw [line width=1pt] (0,0)-- (-1.5,1);
    \begin{scriptsize}
    \draw [fill=black] (0,2) circle (1.8pt);
    \draw [fill=black] (2,2) circle (1.8pt);
    \draw [fill=black] (2,0) circle (1.8pt);
    \draw [fill=black] (0,0) circle (1.8pt);
    \draw [fill=black] (-1.5,0) circle (1.8pt);
    \draw [fill=black] (-1.5,1) circle (1.8pt);
    \end{scriptsize}
    \end{tikzpicture}
    \caption*{\large (v)}
  \end{minipage}
  \caption{Graphs appearing in the $C_4$ casework of the proof of Lemma~\ref{Lemma: diam=2}.}\label{fig:diam2_C4}
\end{figure}

\begin{figure}
  \centering
    \begin{minipage}[b]{0.3\textwidth}\centering
    \begin{tikzpicture}[line cap=round,line join=round,>=triangle 45,x=0.8cm,y=0.8cm]
    \draw [line width=1pt] (-6,0)-- (-5,1);
    \draw [line width=1pt] (-5,1)-- (-4,0);
    \draw [line width=1pt] (-4,0)-- (-6,0);
    \draw [line width=1pt] (-7.5,0)-- (-6,0);
    \begin{scriptsize}
    \draw [fill=black] (-5,1) circle (1.8pt);
    \draw [fill=black] (-4,0) circle (1.8pt);
    \draw [fill=black] (-7.5,0) circle (1.8pt);
    \draw [fill=black] (-6,0) circle (1.8pt);
    \draw[color=black] (-6.1,-0.3) node {$v$};
    \end{scriptsize}
    \end{tikzpicture}
    \caption*{\large (vii)}
  \end{minipage}
      \begin{minipage}[b]{0.3\textwidth}\centering
    \begin{tikzpicture}[line cap=round,line join=round,>=triangle 45,x=0.8cm,y=0.8cm]
    \draw [line width=1pt] (-6,0)-- (-5,1);
    \draw [line width=1pt] (-5,1)-- (-4,0);
    \draw [line width=1pt] (-4,0)-- (-6,0);
    \draw [line width=1pt,color=gray,dash pattern=on 2pt off 2pt] (-6,0)-- (-6,1.5);
    \draw [line width=1pt] (-7.5,0)-- (-6,0);
    \draw [line width=1pt,color=gray,dash pattern=on 2pt off 2pt] (-7,-1)-- (-6,0);
    \draw [line width=1pt,color=gray,dash pattern=on 2pt off 2pt] (-7,-1)-- (-7.5,0);
    \begin{scriptsize}
    \draw [fill=black] (-6,0) circle (1.8pt);
    \draw [fill=black] (-5,1) circle (1.8pt);
    \draw [fill=black] (-4,0) circle (1.8pt);
    \draw [fill=gray] (-7,-1) circle (1.8pt);
    \draw [fill=black] (-7.5,0) circle (1.8pt);
    \draw [fill=gray] (-6,1.5) circle (1.8pt);
    \draw[color=black] (-5.8,-0.3) node {$v$};
    \end{scriptsize}
    \end{tikzpicture}
    \caption*{Possible additions to (vii)}
  \end{minipage} \\
  
      \begin{minipage}[b]{0.3\textwidth}\centering
        \begin{tikzpicture}[line cap=round,line join=round,>=triangle 45,x=0.5cm,y=0.5cm]
        \draw [line width=1pt] (-2,1)-- (-2,-1);
        \draw [line width=1pt] (0,0)-- (-2,-1);
        \draw [line width=1pt] (-2,1)-- (0,0);
        \draw [line width=1pt] (0,0)-- (2,1);
        \draw [line width=1pt] (2,1)-- (2,-1);
        \draw [line width=1pt] (2,-1)-- (0,0);
        \draw [line width=1pt] (0,0)-- (0,2);
        \begin{scriptsize}
        \draw [fill=black] (-2,1) circle (1.8pt);
        \draw [fill=black] (-2,-1) circle (1.8pt);
        \draw [fill=black] (0,0) circle (1.8pt);
        \draw [fill=black] (2,1) circle (1.8pt);
        \draw [fill=black] (2,-1) circle (1.8pt);
        \draw [fill=black] (0,2) circle (1.8pt);
        \end{scriptsize}
        \end{tikzpicture}
    \caption*{\large (viii)}
  \end{minipage}
   \begin{minipage}[b]{0.3\textwidth}\centering
        \begin{tikzpicture}[line cap=round,line join=round,>=triangle 45,x=0.7cm,y=0.7cm]
        \draw [line width=1pt] (-6,0)-- (-5,1);
        \draw [line width=1pt] (-5,1)-- (-4,0);
        \draw [line width=1pt] (-4,0)-- (-6,0);
        \draw [line width=1pt] (-7,0.5)-- (-6,0);
        \draw [line width=1pt] (-7,-0.5)-- (-6,0);        \draw [line width=1pt] (-3,0.5)-- (-4,0);
        \draw [line width=1pt] (-3,-0.5)-- (-4,0);
        \begin{scriptsize}
        \draw [fill=black] (-5,1) circle (1.8pt);
        \draw [fill=black] (-4,0) circle (1.8pt);
        \draw [fill=black] (-7,-0.5) circle (1.8pt);
        \draw [fill=black] (-7,0.5) circle (1.8pt);
        \draw [fill=black] (-3,-0.5) circle (1.8pt);
        \draw [fill=black] (-3,0.5) circle (1.8pt);
        \draw [fill=black] (-6,0) circle (1.8pt);
        \end{scriptsize}
        \end{tikzpicture}
    \caption*{ Jump graph of (viii)}
  \end{minipage}
  \caption{Graphs appearing in the $C_3$ casework of the proof of Lemma~\ref{Lemma: diam=2}.}\label{fig:diam2_C3}
\end{figure}

Next, we assume that $G$ has $C_3$ as a subgraph and no $C_4$. Note that $\text{diam}(C_3) = 1 \neq 2$ so $G$ must strictly contain $C_3$. This implies $G$ has graph (vii) in Figure~\ref{fig:diam2_C3} as a subgraph. We consider what additions we can make to (vii) which preserve a diameter of $2$ and do not create a cycle of length $4$ or more. With these conditions, we are limited to adding pendant edges to vertex $v$ or adding an edge between the leaf vertices of two pendants to create another $C_3$. These additions are shown in Figure~\ref{fig:diam2_C3}. If no additional $C_3$'s are made, then $G$ is in family III in Figure~\ref{fig:dissipate_diam2} and it dissipates. If one additional $C_3$ is made and there are no pendant edges, then $G$ is the Bowtie Graph in Figure~\ref{fig:dissipate_diam2}. In all other cases, $G$ must have graph (viii) in Figure~\ref{fig:diam2_C3} as a subgraph. The jump graph of (iii) contains the Stickman as a subgraph so it has infinite $d$ value and accumulates $C_5$. Hence, in this case $G$ also has $d(G) = \infty$ and accumulates $C_5$.
%
%
\end{proof}

\section{Culminating Results}\label{Section 5}

In Subsection \ref{Section 4.3}, we showed that every connected $d$-infinite graph of diameter 2 or more will accumulate $C_5$ or $N$. In the following theorem, we will bring Lemmas \ref{Lemma: diam>=5}, \ref{Lemma: diam=4}, \ref{Lemma: diam=3}, and \ref{Lemma: diam=2} together to show that every $d$-infinite graph accumulates $C_5$ or $N$.

\begin{thm}[Accumulation Theorem] \label{Theorem: accumulating} 
A graph $G$ is $d$-infinite if and only if $G$ accumulates $C_5$ or $N$.
\end{thm}

\begin{proof}
First, suppose that $G$ accumulates $C_5$ or $N$. Then there is some $k \geq 1$ such that $J^k(G)$ contains $C_5$ or $N$ as a subgraph. It follows directly from Corollary~\ref{Corollary: snipped_d=infinity} that $d(J^k(G)) = \infty$ and consequently, $d(G) = \infty$.

Next, assume that $G$ is $d$-infinite. We claim that it must accumulate $C_5$ or $N$. Suppose first that $G$ is connected. If $G$ has a diameter of 2 or more, then by Lemmas~\ref{Lemma: diam>=5}, \ref{Lemma: diam=4}, \ref{Lemma: diam=3}, and \ref{Lemma: diam=2}, it will accumulate $C_5$ or $N$. If $\text{diam}(G) = 1$, then $G = K_n$ for some $n$. For $n < 5$, we know $d(K_n)$ is finite as can be checked through explicit calculation (see Figure~\ref{fig:dissipating_graphs}). For $n \geq 5$, we observe that $K_n$ has $C_5$ as a subgraph. Thus, if $G$ has diameter $1$ and $d(G) = \infty$, it must have $C_5$ as a subgraph. Therefore, if $G$ is a connected graph and $d(G) = \infty$, then $G$ accumulates $C_5$ or $N$.

Now suppose $G$ is not a connected graph so it has at least two connected components. We will assume $G$ has no isolated vertices, implying that each component contains an edge. Let the connected components of $G$ be $\{H_i\}_{i = 1}^n$, such that $H_i \cap H_j = \emptyset$ if $i \neq j$ and $\bigcup_{i = 1}^n H_i = G$. We show that $J(G)$ is a connected graph.

Let $e_1$ and $e_2$ be vertices in $J(G)$ and we find a path from $e_1$ to $e_2$ in $J(G)$. Suppose first that the edges $e_1$ and $e_2$ are in different components in $G$ such that $e_1 \in H_i$ and $e_2 \in H_j$ for $i \neq j$. Then, $e_1$ and $e_2$ are non-incident in $G$ so the edge $\{e_1, e_2\}$ exists in $J(G)$ and this forms a path of length 1 from $e_1$ to $e_2$. Next, suppose that the edges $e_1$ and $e_2$ are both in $H_k$ for some $k$. Let $H_j$ be another component ($j \neq k)$ and let $e_j$ be an edge in $H_j$. Then $e_j$ is non-incident to both $e_1$ and $e_2$ so the edges $\{e_1, e_j\}$ and $\{e_j, e_2\}$ are in $J(G)$. These edges form a path of length 2 from vertex $e_1$ to vertex $e_2$ in $J(G)$. Therefore, $J(G)$ is a connected graph. By the argument above, $J(G)$ must accumulate $C_5$ or $N$ so $G$ does as well.
\end{proof}

Now that we know every $d$-infinite graph accumulates $C_5$ or $N$, we can ask about the end behavior of the sequence $\{J^k(G)\}$ for a given graph $G$. We know that if $G$ is $C_5$ or $N$, then the number of edges in $J^k(G)$ will stay constant as $k \to \infty$. But what happens for a $d$-infinite graph which is not $C_5$ or $N$? Is there such a $G$ where the number of edges in $J^k(G)$ is constant? Or is there some $G$ where $J^K(G) = G$ for $K > 1$, resulting in a cycle of iterated jump graphs? As it turns out, the answer to both these questions is ``no''. Every $d$-infinite graph which is not $C_5$ or $N$ will grow without bound under the jump graph operation. We will spend the rest of this paper proving this fascinating result.

\begin{lem} \label{Lemma: c5-P_strict_subgraph}
Suppose $G$ has $C_5$ or the net graph $N$ as a strict subgraph, then $\abs{E(J^k(G))} \to \infty$ as $k \to \infty$. \end{lem}

\begin{proof}
First, suppose that $G$ has $C_5$ as a strict subgraph. We claim that the number of edges in $J^k(G)$ cannot decrease. Notice that every edge in $G$ must be non-incident to at least one edge in the $C_5$. Hence, for each edge in $G$, we add at least 1 to the edge count in $J(G)$. Therefore, $\abs{E(G)} \leq \abs{E(J(G))}$. If $G$ has $C_5$ as a strict subgraph, $J(G)$ will as well and so we apply this logic iteratively to conclude $\abs{E(J^{k}(G))} \leq \abs{E(J^{k+1}(G))}$ for all $k \geq 1$. To show that $\abs{E(J^k(G))} \to \infty$ as $k \to \infty$, we start by observing that if $G$ has $C_5$ as a strict subgraph, then $G$ must have the graph $\Gamma$ depicted in Figure~\ref{fig:G0} as a snipped subgraph.

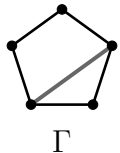
\begin{figure}[H]
    \centering
    \begin{tikzpicture}[line cap=round,line join=round,>=triangle 45,x=0.7cm,y=0.7cm]
    
    \draw [line width=1pt] (0,1)-- (-0.951,0.309);
    \draw [line width=1pt] (-0.951,0.309)-- (-0.588,-0.809);
    \draw [line width=1pt] (-0.588,-0.809)-- (0.588,-0.809);
    \draw [line width=1pt] (0.588,-0.809)-- (0.951,0.309);
    \draw [line width=1pt] (0.951,0.309)-- (0,1);
    \draw [line width=1pt] (0,1)-- (0,1);
    \draw [line width=1.5pt, color=gray] (-0.588,-0.809)-- (0.951,0.309);
    \begin{scriptsize}
    \draw [fill=black] (0,1) circle (1.8pt);
    \draw [fill=black] (-0.951,0.309) circle (1.8pt);
    \draw [fill=black] (-0.588,-0.809) circle (1.8pt);
    \draw [fill=black] (0.588,-0.809) circle (1.8pt);
    \draw [fill=black] (0.951,0.309) circle (1.8pt);
    \draw [color=black] (0, -1.5) node {\large $\Gamma$}; 
    \end{scriptsize}
    \end{tikzpicture}
    \caption{If $G$ has $C_5$ as a strict subgraph, it has $\Gamma$ as a snipped subgraph.}\label{fig:G0}
\end{figure}

Applying Lemma~\ref{Lemma: H_snipped_of_G, J(H)_in_J(G)} iteratively, we can conclude that $J^k(\Gamma) \subseteq J^k(G)$ for all $k \geq 1$. We will show that $\abs{E(J^k(\Gamma))} \to \infty$ as $k \to \infty$, and then this must also be true for $G$.

To do this, we will prove that for all $k \geq 1$, 
\begin{equation}\label{eq:induction1}
J^k(\Gamma) \subseteq J^{k + 1}(\Gamma)
\end{equation}
and
\begin{equation}\label{eq:induction2}
    \abs{E(J^{k + 1}(\Gamma))} \ge \abs{E(J^k(\Gamma))} + 2.
\end{equation}

We proceed by induction. The first and second jump graphs of $\Gamma$ are shown in Figure~\ref{fig:G0_m(G0)_m2(G0)}; these establish \eqref{eq:induction1} and \eqref{eq:induction2} for $k=1$.

Now assume that \eqref{eq:induction1} and \eqref{eq:induction2} hold for $k$ and we show they also hold for $k + 1$. Because $J^k(\Gamma)~\subseteq~J^{k + 1}(\Gamma)$, we apply Lemma~\ref{Lemma: H_snipped_of_G, J(H)_in_J(G)} to conclude that $J^{k + 1}(\Gamma) \subseteq J^{k + 2}(\Gamma)$ and this establishes that \eqref{eq:induction1} holds for $k+1$. Let $H$ be a subgraph of $J^{k+1}(\Gamma)$ that is isomorphic to $J^k(\Gamma)$ and let $E'$ be the set of all edges in $J^{k + 1}(\Gamma)$ which are not in $H$. As in the discussion in the first paragraph of this proof, every edge in $E'$ must add at least 1 to the total edge count of $J^{k+2}(\Gamma)$. Because equation \eqref{eq:induction2} holds for $k$, we know $\abs{E'} \geq 2$ and this gives
\[\abs{E(J^{k + 2}(\Gamma))} \geq \abs{E(J(H))} + \abs{E'} = \abs{E(J^{k + 1}(\Gamma))} + \abs{E'} = \abs{E(J^{k + 1}(\Gamma))} + 2\]
which establishes \eqref{eq:induction2} for $k+1$. By induction, we conclude that \eqref{eq:induction1} and \eqref{eq:induction2} are true for all $k\ge 1.$ Since the number of edges in the $k^{\text{th}}$ jump graph of $\Gamma$ is a strictly increasing sequence of natural numbers for $k \geq 1$, it must be unbounded. Hence, $\abs{E(J^k(\Gamma))} \to \infty$ as $k \to \infty$.

\begin{figure}[H]
  \centering
    \begin{minipage}[b]{0.3\textwidth}\centering
    \begin{tikzpicture}[line cap=round,line join=round,>=triangle 45,x=0.7cm,y=0.7cm]
    \draw [line width=1pt] (0,1)-- (-0.951,0.309);
    \draw [line width=1pt] (-0.951,0.309)-- (-0.588,-0.809);
    \draw [line width=1pt] (-0.588,-0.809)-- (0.588,-0.809);
    \draw [line width=1pt] (0.588,-0.809)-- (0.951,0.309);
    \draw [line width=1pt] (0.951,0.309)-- (0,1);
    \draw [line width=1pt] (0,1)-- (0,1);
    \draw [line width=1.5pt,color=gray] (-0.588,-0.809)-- (0.951,0.309);
    \begin{scriptsize}
    \draw [fill=black] (0,1) circle (1.8pt);
    \draw [fill=black] (-0.951,0.309) circle (1.8pt);
    \draw [fill=black] (-0.588,-0.809) circle (1.8pt);
    \draw [fill=black] (0.588,-0.809) circle (1.8pt);
    \draw [fill=black] (0.951,0.309) circle (1.8pt);
    \end{scriptsize}
    \end{tikzpicture}
    \caption*{\large $\Gamma$}
  \end{minipage}
    \begin{minipage}[b]{0.3\textwidth}\centering
    \begin{tikzpicture}[line cap=round,line join=round,>=triangle 45,x=0.7cm,y=0.7cm]
    \draw [line width=1pt] (5,1)-- (4.05,0.309);
    \draw [line width=1pt] (4.05,0.309)-- (4.42,-0.809);
    \draw [line width=1pt] (4.42,-0.809)-- (5.588,-0.809);
    \draw [line width=1pt] (5.588,-0.809)-- (5.951,0.309);
    \draw [line width=1pt] (5.951,0.309)-- (5,1);
    \draw [line width=1.5pt,color=gray] (6.680,0.680)-- (5.951,0.309);
    \begin{scriptsize}
    \draw [fill=black] (5,1) circle (1.8pt);
    \draw [fill=black] (4.05,0.309) circle (1.8pt);
    \draw [fill=black] (4.42,-0.809) circle (1.8pt);
    \draw [fill=black] (5.588,-0.809) circle (1.8pt);
    \draw [fill=black] (5.951,0.309) circle (1.8pt);
    \draw [fill=gray] (6.680,0.680) circle (1.8pt);
    \end{scriptsize}
    \end{tikzpicture}
    \caption*{\large $J(\Gamma)$}
  \end{minipage}
     \begin{minipage}[b]{0.3\textwidth}\centering
        \begin{tikzpicture}[line cap=round,line join=round,>=triangle 45,x=0.7cm,y=0.7cm]
        \draw [line width=1pt] (10,1)-- (9.05,0.309);
        \draw [line width=1pt] (9.05,0.309)-- (9.42,-0.809);
        \draw [line width=1pt] (9.42,-0.809)-- (10.588,-0.809);
        \draw [line width=1pt] (10.588,-0.809)-- (10.951,0.309);
        \draw [line width=1pt] (10.951,0.309)-- (10,1);
        \draw [line width=1.5pt,color=gray] (10, 0)-- (10.951,0.309);
        \draw [line width=1.5pt, color=gray] (10,0)-- (9.05,0.309);
        \draw [line width=1.5pt, color=gray] (10,0)-- (9.42,-0.809);
        \begin{scriptsize}
        \draw [fill=black] (10,1) circle (1.8pt);
        \draw [fill=black] (9.05,0.309) circle (1.8pt);
        \draw [fill=black] (9.42,-0.809) circle (1.8pt);
        \draw [fill=black] (10.588,-0.809) circle (1.8pt);
        \draw [fill=black] (10.951,0.309) circle (1.8pt);
        \draw [fill=gray] (10,0) circle (1.8pt);
    \end{scriptsize}
    \end{tikzpicture}
    \caption*{\large $J^2(\Gamma)$}
  \end{minipage}
  \caption{Graph $\Gamma$ and its first and second jump graphs.}\label{fig:G0_m(G0)_m2(G0)}
\end{figure}

In the other case, suppose that $G$ has $N$ as a strict subgraph and consider the number of edges in $J(G)$. We know $\abs{E(J(G))} \geq 6$ because $N \subseteq J(G)$. Now let $E' \subseteq E(G)$ be the set of edges in $G$ which are not in the subgraph $N$. Every edge $e \in E'$ must be non-incident to at least two edges in $N \subseteq G$ because of the structure of $N$. Hence, every $e \in E'$ will add at least $2$ to the total edge count of $J(G)$. This observation, and the fact that $\abs{E(G)} = \abs{E'} + 6$, means that 
\[\abs{E(J(G))} \geq 2\abs{E'} + \abs{E(N)} = 2\abs{E'} + 6 = \abs{E'} + \abs{E(G)}.\]
If $G$ has $N$ as a \emph{strict} subgraph, then $\abs{E'} \geq 1$ and so we have $\abs{E(J(G))} \geq \abs{E(G)} + 1$. Then notice that $J(G)$ must also have $N$ as a strict subgraph and we apply this logic again. Following this reasoning iteratively, we have $\abs{E(J^{k + 1}(G))} \geq \abs{E(J^k(G))} + 1$ for all $k \geq 1$. The number of edges in $J^k(G)$ is then a strictly increasing sequence of natural numbers, implying that it must be unbounded. Hence, we have $\abs{E(J^k(G))} \to \infty$ as $k \to \infty$.
\end{proof}

\begin{cor}
If $G$ has $C_5$ or the net graph $N$ as a snipped subgraph and $G \neq C_5$, $G \neq N$, then $\abs{E(J^k(G))} \to \infty$ as $k \to \infty$.
\end{cor}

\begin{proof}
If $G$ has $C_5$ as a snipped subgraph, then $J(G)$ will have $C_5$ as a subgraph by Lemma~\ref{Lemma: H_snipped_of_G, J(H)_in_J(G)}. Assuming that $G$ differs from $C_5$ by more than isolated vertices, we know that $J(G) \neq J(C_5) = C_5$. Hence, we conclude that $J(G)$ has $C_5$ as a strict subgraph. Applying Lemma~\ref{Lemma: c5-P_strict_subgraph}, we come to the desired conclusion. The same argument can be applied to $N$.
\end{proof}

\begin{thm}[Exploding Graph Theorem] \label{Theorem: Exploding graph}
If $G$ has infinite $d$ value and $G$ is not $C_5$ or the net graph $N$, then $\abs{E(J^k(G))} \to \infty$ as $k \to \infty$.
\end{thm}

\begin{proof}
Suppose that $G$ is a graph and $d(G) = \infty$. Then, by the Accumulation Theorem (\ref{Theorem: accumulating}), we know that $G$ accumulates $C_5$ or the net graph $N$. Since $G$ is not either of these graphs, there must be some $k$ such that $J^k(G)$ has $C_5$ or $N$ as a strict subgraph. Hence, by Lemma \ref{Lemma: c5-P_strict_subgraph}, we know that $\abs{E(J^k(G))} \to \infty$ as $k \to \infty$.
\end{proof}

\begin{thm} \label{Theorem: J^k(G)=G}
For any non-empty graph $G$, the following are equivalent.
\begin{enumerate}[label=(\roman*)]
    \item $J(G) = G$.
    \item $J^k(G) = G$ for some $k$.
    \item $G$ is $C_5$ or $N$.
\end{enumerate}
\end{thm}

\begin{proof}
The  implication $(i) \Rightarrow (ii)$ is immeadiate and we have already discussed the implication $(iii) \Rightarrow (i)$. Hence, we only need to show that $(ii)$ implies $(iii)$. Note that if $(ii)$ is true, then we know $J^{nk}(G) = G$ for all positive integers $n$.

Suppose that $J^{k_0}(G) = G$ for some positive integer ${k_0}$. If there is some $d$ such that $J^d(G) = \emptyset$, then we can find a positive integer $n$ such that $d \leq nk_0$. This means that that $J^{nk_0}(G) = \emptyset$ but also $J^{nk_0}(G) = G$ by assumption, providing a contradiction. Hence, we know that $d(G) = \infty$. The Accumulation Theorem (\ref{Theorem: accumulating}) implies that $G$ will accumulate $C_5$ or $N$. Suppose that $G$ accumulates $C_5$ or $N$ as a strict subgraph. By Lemma~\ref{Lemma: c5-P_strict_subgraph}, we see that the number of edges must increase without bound as we continue taking jump graphs. This produces a contradiction because $J^{nk_0}(G) = G$ for any positive integer $n$ so the sequence $\{\abs{E(J^k(G))}\}$ must be repeating and therefore bounded. Hence, $G$ does not have $C_5$ or $N$ as a strict subgraph but still accumulates one of the two graphs. This implies that $G = C_5$ or $G = N$.
\end{proof}

\begin{cor}
If a graph $G$ is not $C_5$ or the net graph $N$, then there is no $k \geq 1$ such that \\ $J^k(G) = G$. 
\end{cor}

The above is an immediate corollary of Theorem~\ref{Theorem: J^k(G)=G}.

\section*{Acknowledgements}

We would like to explicitly express gratitude to the WXML (Washington Experimental Mathematics Lab) at the University of Washington. This program brought us together with our outstanding mentors, Bennet Goeckner and Rowan Rowlands, with Rowan supplying the excellent images in Figure 5, Figure 6, and Figure 7. They provided extensive support in both pursuing our work and writing this paper, offering feedback, proofreading, and more. Their efforts were not only in providing us with necessary or interesting background information but also in pushing us to do more, especially with their shared excitement as we reported progress. Our work here certainly would not have happened without them.

\bibliography{bib}

\providecommand{\bysame}{\leavevmode\hbox to3em{\hrulefill}\thinspace}
\providecommand{\MR}{\relax\ifhmode\unskip\space\fi MR }
\providecommand{\MRhref}[2]{%
  \href{http://www.ams.org/mathscinet-getitem?mr=#1}{#2}
}
\providecommand{\href}[2]{#2}
\begin{thebibliography}{CHJS97}

\bibitem[Aig69]{Aig}
Martin Aigner, \emph{Graphs whose complement and line graph are isomorphic},
  Journal of Combinatorial Theory \textbf{7} (1969), no.~3, 273--275.

\bibitem[Bei70]{Bein}
Lowell~W Beineke, \emph{Characterizations of derived graphs}, Journal of
  Combinatorial Theory \textbf{9} (1970), no.~2, 129--135.

\bibitem[CHJS97]{Char}
Gary Chartrand, H\'{e}ctor Hevia, Elzbieta Jarett, and Michelle Schultz,
  \emph{Subgraph distances in graphs defined by edge transfers}, Discrete
  Mathematics \textbf{170} (1997), no.~1-3, 63--79.

\bibitem[Jon08]{Jon}
Jakob Jonsson, \emph{Simplicial complexes of graphs}, Lecture Notes in
  Mathematics, vol. 1928, Springer-Verlag, Berlin, 2008.

\bibitem[Ste]{Steinbach}
Peter Steinbach, \emph{Field guide to simple graphs, {V}olume 4, {P}art 3.
  \emph{Available at \url{http://oeis.org/A000664}.}}

\bibitem[Wac03]{Wachs}
Michelle~L. Wachs, \emph{Topology of matching, chessboard, and general bounded
  degree graph complexes}, vol.~49, 2003, Dedicated to the memory of Gian-Carlo
  Rota, pp.~345--385.

\bibitem[Wes96]{West}
Douglas~B. West, \emph{Introduction to graph theory}, Prentice Hall, Inc.,
  Upper Saddle River, NJ, 1996.

\bibitem[Whi92]{Whit}
Hassler Whitney, \emph{Congruent graphs and the connectivity of graphs},
  Hassler Whitney Collected Papers, Springer, 1992, pp.~61--79.

\end{thebibliography}
\bibliographystyle{amsalpha}

\end{document}